\documentclass[12pt]{article}
\usepackage[dvips]{graphicx}
\usepackage{color}
\usepackage{latexsym,amsmath,amsfonts,amscd}
\usepackage{amssymb,graphicx,amsthm}
\usepackage{epsfig}
\usepackage{mathrsfs}
\usepackage{indentfirst}
\usepackage{pstricks}
\usepackage{pst-plot}
\usepackage{multirow}
\usepackage{subfigure}
\usepackage{psfrag}
\usepackage{caption}
\usepackage{enumerate}
\usepackage{float}
\usepackage{bbm}
\usepackage{bm}

\allowdisplaybreaks
\newtheorem{theorem}{Theorem}[section]

\newtheorem{remark}{Remark}[section]

\numberwithin{figure}{section}
\numberwithin{table}{section}

\begin{document}

\baselineskip=2pc
\vspace*{.30in}

\begin{center}
{\bf Analysis and Hermite spectral approximation of diffusive-viscous wave equations in unbounded domains arising in geophysics}
\end{center}


\centerline{Dan Ling\footnote{School of Mathematics and Statistics,
Xi'an Jiaotong University, Xi'an, Shaanxi 710049, China.
E-mail: danling@xjtu.edu.cn. Research partially supported by National Natural
Science Foundation of China grant 12101486, China
Postdoctoral Science Foundation grant 2020M683446 and the High-performance
Computing Platform at Xi'an Jiaotong University.} and
Zhiping Mao\footnote{School of Mathematical Sciences, Fujian Provincial Key Laboratory of Mathematical Modeling and High-Performance Scientific Computing,  Xiamen University, Xiamen, Fujian 361005, China. E-mail:
zpmao@xmu.edu.cn. Research partially supported by the Fundamental Research Funds for the Central Universities (20720210037).}
}

\vspace{.25in}

\centerline{\bf Abstract}

\bigskip

\baselineskip=1.4pc

The diffusive-viscous wave equation (DVWE) is widely used in seismic exploration since it can explain frequency-dependent seismic reflections in a reservoir with hydrocarbons. Most of the existing numerical approximations for the DVWE are based on domain truncation with ad hoc boundary conditions. However, this would generate artificial reflections as well as truncation errors. To this end, we directly consider the DVWE in unbounded domains. We first show the existence, uniqueness, and regularity of the solution of the DVWE. We then develop a Hermite spectral Galerkin scheme and derive the corresponding error estimate showing that the Hermite spectral Galerkin approximation delivers a spectral rate of convergence provided sufficiently smooth solutions. Several numerical experiments with constant and discontinuous coefficients are provided to verify the theoretical result and to demonstrate the effectiveness of the proposed method. In particular, We verify the error estimate for both smooth and non-smooth source terms and initial conditions.  In view of the error estimate and the regularity result, we show the sharpness of the convergence rate in terms of the regularity of the source term. We also show that the artificial reflection does not occur by using the present method.

\vspace{.05in}

\vfill


\newpage

\baselineskip=2pc

\section{Introduction}

The numerical simulation of wave propagation in media with solid and fluid layers plays an important role in seismic exploration data analysis.
It has been found that seismic reflections are
frequency-dependent \cite{Geertsma1961,Brown2009,Chapman2009}.
And, the frequency-dependent reflections from a fluid-saturated  porous medium is relatively complex.
For instance, it has been shown from both laboratory analysis and field data that for the fluid-saturated layer, the resulting reflections have a higher amplitude and delayed travel-time at low-frequencies when compared with the reflections from a gas-saturated layer \cite{korneev}.
However, this important phenomenon cannot be well described by the Biot's theory \cite{biot2,biot3,goloshubin}.
Moreover, the acoustic and elastic theories are unable to effectively characterize the subsurface in fluid-saturated rocks \cite{korneev}.
Therefore, to develop more accurate theoretical models and make a wider application in practical seismic exploration, a Diffusive-Viscous Wave Equation (DVWE) was proposed in \cite{korneev} by adding a diffusive dissipation term and a viscous term to the scalar wave equation to study the connection between fluid saturation and frequency dependence of reflections and to characterize the attenuation property of the seismic wave in a fluid-saturated medium.

Recently, researchers from both scientific and industrial communities have paid much attention to the study on DVWEs. By means of the Biot's theory, Quintal et al. \cite{quintal} proposed an interlayer-flow model, which was approximated by using a finite difference scheme, to study the reflections in the low-frequency range providing a physical basis to the diffusive-viscous theory as well as explaining the spectral anomalies observed at low frequencies in thinly layered reservoirs.
To simulate the frequency-dependent seismic response of turbidite reservoirs, a seismic data-driven geological model was first applied to produce physical parameter sections, which were then used to numerically synthesize the frequency-dependent seismic response of turbidite reservoirs by simulating the DVWEs~\cite{chen}.
Zhao et al. proposed finite difference methods to simulate wave-fields of DVWEs in \cite{zhao1,zhao2} and applied the reflectivity method for the numerical modeling of DVWEs in layered medium in \cite{zhao3}, the analysis of the von Neumann stability criteria and the numerical dispersion were given in \cite{zhao1}.
A finite volume method was developed in \cite{mensah} to simulate the seismic wave propagation in a fluid-saturated medium driven by the DVWE. Recently, Ling et al. proposed the local discontinuous Galerkin method and analyzed the error estimates for the DVWEs with variable coefficients in \cite{ling2023}.
%
More work can be found in \cite{he,han2,zhao2022} and references therein.

However, most existing works only focus on the numerical schemes and the corresponding results of the stability and error estimates, there is few theoretical work concerning the existence and uniqueness of the solution for DVWEs.
Most recently, we noticed an important work provided by Han et al. in \cite{han}, in which the well-posedness and stability of DVWEs were established in bounded domains, this provided a theoretical foundation to develop numerical methods. However, no regularity result is discussed in the work.
In addition, as mentioned in \cite{carcione}, another main aspect for DVWEs is the non-reflection boundary conditions.
All aforementioned numerical simulations were performed in bounded domains, which may generate artificial reflections due to the truncation of the model. To resolve this issue, a non-split perfectly matched layer boundary condition was proposed for the DVWE to absorb the artificial reflections~\cite{zhao2022} (see also \cite{zeng2001, liu2010,rao2016}). However, the resulted problem is more complex in terms of computer implementation and it is computationally more expensive.

The aim of this work is to consider the DVWE without the truncation of the domain, i.e., we consider the DVWE directly in unbounded domains, and then establish the existence and uniqueness of the weak solution. We also discuss the regularity of the solutions in terms of the initial conditions and the source term. Furthermore, we develop an efficient Hermite spectral Galerkin scheme to approximate the solution of the DVWE.
To this end, we consider in this work the following DVWE
\begin{equation}\label{eq1}
\partial_t^2u+\alpha \partial_t u
-\partial_t{\rm div}(\beta\nabla u)-{\rm div}(\gamma^2\nabla u)=f, ~~~\bm{x}\in\mathbb{R}^d,~~t>0
\end{equation}
subjecting to following initial conditions
\begin{equation}\label{eq2}
u(\bm{x},0)=u_0(\bm{x}),~~~\partial_tu(\bm{x},0)=w_0(\bm{x}),
\end{equation}
where $d$ is the dimension in space, $u=u(\bm{x},t)$ is the wave field, $\alpha=\alpha(\bm{x})$ and $\beta=\beta(\bm{x})$ are the
diffusive and viscous attenuation parameters respectively, $\gamma=\gamma(\bm{x})$ is the wave
propagation speed in the non-dispersive medium, $f=f(\bm{x},t)$ is the source.
In this paper, we consider the cases of $d=1,2$ and assume that
\begin{equation}\label{eq:abc:cond}
    \begin{aligned}
    &\alpha(\bm{x})\in L^\infty(\Omega),\; 0< \alpha_1 \le \alpha(\bm{x}) \le \alpha_2, \;
    \beta(\bm{x})\in L^\infty(\Omega),\; 0< \beta_1 \le \beta(\bm{x}) \le \beta_2, \\
    &\gamma(\bm{x})\in L^\infty(\Omega),\; 0< \gamma_1 \le \gamma(\bm{x}) \le \gamma_2.
    \end{aligned}
\end{equation}
We apply the Hermite spectral method since it has two main advantages:
\begin{itemize}
    \item The first advantage is that the Hermite spectral method is a natural choice to deal with unbounded domain problems, see \cite{guo1999,ma2005,mao,shen2009,xiang2010}.
    \item The second one is the Hermite spectral method enjoys high accuracy provided that the solution is smooth enough.
\end{itemize}
We also derive the error estimate for the Hermite spectral Galerkin approximation showing that it delivers a spectral rate of convergence provided sufficiently smooth solution. To the best of our knowledge, this is the first attempt that DVWEs are analyzed and solved in the unbounded domains.

The remainder of this paper is organized as follows. In Section 2, we provide some preliminaries about Hermite orthogonal polynomials and functions and the corresponding approximation results. In Section 3, we give the weak form and the Hermite spectral Galerkin approximation for the problem \eqref{eq1}, and establish the existence and uniqueness of the weak solutions and discuss the regularity of the solution. We derive the error estimates in Section 4. In Section 5,
we present several numerical examples to demonstrate the convergence and effectiveness of the presented Hermite spectral Galerkin methods.
Finally, we give some concluding remarks in Section 6.

\section{Preliminary}\label{sec:pre}
In this section, we first introduce the Hermite orthogonal functions and the corresponding approximation results.

Let $\bm{x}=(x_1,\cdots,x_d)$ denote the  multi-variable in $\Omega:=\mathbb{R}^d$. For any function $u(\bm{x})\in L^2(\Omega)$, we denote its Fourier transform as $\widehat{u}(\bm{\xi})$. $|\bm{\xi}|_1,|\bm{\xi}|_2$ and $|\bm{\xi}|_\infty$ stand for the $l^1, l^2$ and $l^\infty$ norm of $\bm{\xi}$
in $\mathbb{R}^d$, respectively.
Let $\omega(\bm x)>0\;(\bm x\in \Omega)$ be a weight function, we denote $L_\omega^2(\Omega)$ the usual weighted Hilbert space with the
 inner product and norm defined by
\begin{equation*}
 (u,v)_{\Omega,\omega}= \int_{\Omega}u(\bm x)v(\bm x)\omega(\bm x)\, d \bm x,\;\; \|u\|_{\Omega,\omega}=(u,u)_{\Omega,\omega}^{\frac{1}{2}},\;
 \forall\, u,v \in L_\omega^2(\Omega).
\end{equation*}
When $\omega\equiv 1$, we will drop $\omega$ from the above notations.
The Plancherel Theorem states that
\begin{equation*}
    \|u\|_{\Omega} = \|\widehat{u}\|_{\Omega}.
\end{equation*}

We  denote by $H^{\mu}(\Omega)$ (with $\mu\ge 0$) the usual Hilbert spaces with semi-norm
\begin{equation*}
    |u|_{\mu,\Omega} = \||\bm{\xi}|_2^{\mu}  \widehat{u}\|_{\Omega}
\end{equation*}
and norm
\begin{equation*}
    \|u\|_{\mu,\Omega} = (\|u\|_{\Omega}^2 + |u|_{\mu,\Omega}^2)^{1/2}
    = (\|\widehat{u}\|_{\Omega}^2 + \| |\bm{\xi}|_2^{\mu}  \widehat{u}\|_{\Omega}^2)^{1/2}.
\end{equation*}

Let $c$ be a generic positive constant independent of any functions and of any discretization
parameters. We use the expression $A\lesssim B $ (respectively $A\gtrsim B $) to mean that $A \leqslant cB $ (respectively $A\geqslant cB $), and use the expression $A\cong B $ to mean that $A \lesssim B \lesssim A$. We will also drop $\Omega$ or $\mathbb{R}^d$ from the notations if no confusion arises.

We first introduce the orthonormal Hermite polynomials $\{H_n(x)\}$ in $\mathbb{R}$, which are defined by the three-term recurrence relation:
\begin{equation*}
\begin{aligned}
    &H_{n+1}(x) = x\sqrt{\frac{2}{n+1}} H_n(x)- \sqrt{\frac{n}{n+1}}H_{n-1}(x),\; n\geq 1,\\
    &H_0(x) = \pi^{-1/4}, \quad H_1(x) = \sqrt{2}\pi^{-1/4}x.
\end{aligned}
\end{equation*}
They are mutually orthogonal with respect to the weight function $\omega(x) = {\rm e}^{-x^2}$, i.e.,
\begin{equation}\label{Her-poly}
    \int_{-\infty}^{\infty} H_m(x)H_n(x)\omega(x)dx = \delta_{mn},
\end{equation}
and it satisfies that
\begin{equation*}
    H'_{n}(x) = \sqrt{2n} H_{n-1}(x),\; n\geq 1.
\end{equation*}


Denote $P_N(x)$ the space of the polynomials of degree at most $N$, and we have
\begin{equation*}\label{eq4}
P_N(x)=\text{span}\big\{H_0(x),H_1(x),\cdots,H_N(x)\big\}.
\end{equation*}
Let  $P_N^d$ be the $d$ dimension tensor of $P_N$.
We define the orthogonal projection $\bm{\Pi}_{\bm N}: L_{\bm \omega}^2(\mathbb{R}^d) \rightarrow P_N^d$,
\begin{equation}\label{L2-pro-multi}
    \int_{\mathbb{R}^d} (\bm{\Pi}_{N}u - u)v_N {\bm \omega} (\bm x)dx  =0,\quad \forall \, v_N \in P_N^d,
\end{equation}
where ${\bm \omega}(\bm x) = \prod_{j=1}^d \omega(x_j).$

Let us introduce the Hermite orthogonal functions
\begin{align*}
    \phi_j(x) = {\rm e}^{-x^2/2}H_j(x),\, j =0,1,\ldots,
\end{align*}
which form an orthogonal basis in $L^2(\mathbb{R})$, i.e.,
\begin{equation*}
    \int_{-\infty}^{\infty} \phi_m(x)\phi_n(x)dx = \delta_{mn}
\end{equation*}
according to \eqref{Her-poly}.
Let
\begin{equation*}
\mathcal{P}_N(x)=\big\{{\rm e}^{-\frac{x^2}{2}}v~|~v\in P_N(x)\big\} = \text{span}\big\{\phi_0(x), \phi_1(x),\cdots,\phi_N(x)\big\},
\end{equation*}
and denote $V_N$ the $d$ dimension tensor product of $\mathcal{P}_N$.
We next consider approximations by multivariate Hermite functions.
Note that for any $u\in L^2(\mathbb{R}^d)$, we have $u \bm{\omega}^{-1/2}\in L_{\bm{\omega}}^2(\mathbb{R}^d)$.
Define
\begin{equation}\label{PiNhat}
    \hat{\bm{\Pi}}_N u := \bm{\omega}^{1/2} \bm{\Pi}_N(u\bm{\omega}^{-1/2})\in {V_N}.
\end{equation}
Then for $u\in L^2(\mathbb{R}^d)$, we derive immediately from  \eqref{L2-pro-multi} that
\begin{equation*}
    \int_{\mathbb{R}^d} ( \hat{\bm{\Pi}}_{N}u - u)v_N d\bm x=0,\quad \forall \, {v_N \in V_N}.
\end{equation*}

We introduce the operator $\hat{\partial}_{x_j}={\partial}_{x_j}+x_j,$
which satisfies
\begin{equation*}
   \omega^{-1/2}(x_j)\hat{\partial}_{x_j}u(x_j) = \partial_{x_j} \big[\omega^{-1/2}(x_j)u(x_j)\big],
\end{equation*}
and denote $\hat{\partial}_{\bm x}:= \prod_{j=1}^d \hat{\partial}_{x_j}$,  $\hat{\bm{\partial}}_{\bm x}^{\bm k}:= \prod_{j=1}^d \hat{\partial}_{x_j}^{k_j}$.
Furthermore, we define the following weighted Sobelev space
\begin{equation*}
    \hat{B}^m(\mathbb{R}^d) := \big\{u: \hat{\partial}_{\bm x}^{\bm{k}} u\in L^2(\mathbb{R}^d),\;0\le |\bm k|_1\le m  \big\}, \;\forall m\in \mathbb{N},
\end{equation*}
equipped with the norm and semi-norm
\begin{equation*}
    \|u\|_{\hat{B}^m(\mathbb{R}^d)} = \Big(\sum_{0\le |\bm k|_1 \le m} \|\hat{\partial}_{\bm x}^{\bm{k}} u\|^2 \Big)^{\frac{1}{2}}, \quad
    |u|_{\hat{B}^m(\mathbb{R}^d)} = \Big(\sum_{j=1}^d \|\hat{\partial}_{ x_j}^{m} u\|^2 \Big)^{\frac{1}{2}}.
\end{equation*}
We present below the approximation result for the errors measured in the usual Hilbert space~\cite{mao}.
\begin{theorem}\label{thmErrEstiHerFunMultiHat}
For any $u \in \hat{B}^m(\mathbb{R}^d)$ with $m\ge 1$, we have
\begin{equation}\label{eqnErrEstiHerMultiH}
    \|\bm{\hat\Pi}_{N}u-u\|_{H^\mu(\mathbb{R}^d)} \lesssim N^{(\mu-m)/2}|u|_{\hat B^m(\mathbb{R}^d)},\quad 0\le \mu\le m.
\end{equation}
%
\end{theorem}

\section{Well-posedness and regularity}\label{sec:wellposedness}

The existence and uniqueness of the solution of the DVWE in a bounded domain with mixed boundary conditions are given in \cite{han}. However, there is no result regarding the well-posedness of the diffusive-viscous wave equation in the unbounded domain. In this section we would like to show the existence and uniqueness of the weak solution in the unbounded domain. Moreover, we also discuss the regularity of the solution in terms of the initial conditions and the source term $f$.

\subsection{Existence and uniqueness of the weak solution}\label{subsec:wellposedness}
We first show the existence and uniqueness of the solution of the DVWE. Here We use more or less the standard arguments in \cite[Section 7.2]{Evans}.

Let $H^1(\mathbb{R}^d)$ be the usual Sobolev space. By using the integration by parts, we have the weak form of the problem \eqref{eq1}: For a.e. $t\in(0,T)$, find $u(t),~ {\partial_tu}(t)\in H^1(\mathbb{R}^d), ~ {\partial_t^2 u}(t)\in H^{-1}(\mathbb{R}^d)$, such that
\begin{equation}\label{eq:wk}
A(u,v) = (f,v),\quad \forall v\in H^1(\mathbb{R}^d)
\end{equation}
with $u(\bm{x},0)=u_0(\bm{x}),\;\partial_t u(\bm{x},0)=w_0(\bm{x}),$ where $H^{-1}(\mathbb{R}^d)$ is the dual space of $H^1(\mathbb{R}^d)$, and
\begin{equation}\label{eq22}
A(u,v):=(\partial_t^2u,v)+(\alpha \partial_tu,v)+(\beta\partial_t\nabla u,\nabla v)+(\gamma^2\nabla u,\nabla v).
\end{equation}

The Hermite spectral Galerkin approximation to \eqref{eq:wk} is to find $u_N(t),\partial_tu_N(t)\in V_N$, such that
\begin{equation}\label{eq:sg}
A(u_N,v) = (f,v), \quad \forall v\in V_N
\end{equation}
with $u_N(\bm{x},0)=\hat{\bm{\Pi}}_N u_0,\partial_t u_N(\bm{x},0)=\hat{\bm{\Pi}}_N w_0$, where $\hat{\bm{\Pi}}_N$ is the $L^2$-projection defined in \eqref{PiNhat}.

By using the standard arguments for ordinary differential equations (see \cite[Theorem 25.3]{Robinson}), we have that the Galerkin approximation \eqref{eq:sg} admits a unique solution.
Before we prove the well-posedness of the continuous problem \eqref{eq:wk}, we begin by establishing two results on the continuous dependence of the Hermite spectral Galerkin approximation, which will be used to study the well-posedness of the weak problem \eqref{eq:wk}.
\begin{theorem}\label{thm:SG:est}
Assume $u_0\in H^1(\mathbb{R}^d),\; w_0\in L^2(\mathbb{R}^d),\; f\in L^2(0,T;H^{-1}(\mathbb{R}^d))$, then $u_N(t)$ satisfies the following two estimates:
\begin{equation}\label{eq:stab3}
\begin{aligned}
\|\partial_t u_N\|_{L^{\infty}(0,T;L^2(\mathbb{R}^d))} + \|\partial_t u_N\|_{L^2(0,T;H^1(\mathbb{R}^d))} + \|u_N\|_{L^\infty(0,T; H^1(\mathbb{R}^d))}
\le \bar{C}(u_0, w_0, f),
\end{aligned}
\end{equation}
\begin{equation}\label{eq:stab4}
\begin{aligned}
\|\partial_t^2 u_N\|_{L^2(0,T;H^{-1}(\mathbb{R}^d))}
\le \tilde{C}(u_0, w_0, f),
\end{aligned}
\end{equation}
where $\bar{C}$ and $\tilde{C}$ are two constants depending on $\|f\|_{L^2(0,T; H^{-1}(\mathbb{R}^d))}$ and $\|u_0\|_{H^1(\mathbb{R}^d)},\; \|w_0\|_{L^2(\mathbb{R}^d)}$,  but independent of $t$ and $N$.
\end{theorem}
\begin{proof}
By taking $v = \partial_t u_N(t)$ in \eqref{eq:sg} and using the Cauchy-Schwarz and Young inequalities, we have for any $\varepsilon>0$,
\begin{equation*}
\begin{aligned}
\frac{1}{2}\frac{d}{dt}\left(\|\partial_t u_N(t)\|_{L^2(\mathbb{R}^d)}^2 + \|\gamma\nabla u_N(t)\|_{L^2(\mathbb{R}^d)}^2 \right) +\|\alpha^{1/2}\partial_t u_N(t)\|_{L^2(\mathbb{R}^d)}^2 + \|\beta^{1/2}\nabla(\partial_t u_N)(t)\|_{L^2(\mathbb{R}^d)}^2 & \\
= \left ( f, \partial_t u_N(t)\right )
\le \|f\|_{H^{-1}(\mathbb{R}^d)}\cdot \|\partial_t u_N(t)\|_{H^1(\mathbb{R}^d)}
\le \varepsilon \|\partial_t u_N(t)\|_{H^1(\mathbb{R}^d)}^2 + \frac{1}{4\varepsilon} \|f\|_{H^{-1}(\mathbb{R}^d)}^2. &
\end{aligned}
\end{equation*}
Let $\varepsilon = \min\{\alpha, \beta\}/2$, we obtain  from  the first two conditions in \eqref{eq:abc:cond} that
\begin{equation*}
\begin{aligned}
\frac{d}{dt}\left(\|\partial_t u_N(t)\|_{L^2(\mathbb{R}^d)}^2 + \|\gamma\nabla u_N(t)\|_{L^2(\mathbb{R}^d)}^2 \right) +C(\alpha, \beta, \varepsilon)\|\partial_t u_N(t)\|_{H^1(\mathbb{R}^d)}^2
\le \frac{1}{4\varepsilon} \|f\|_{H^{-1}(\mathbb{R}^d)}^2,
\end{aligned}
\end{equation*}
where $C(\alpha, \beta, \varepsilon)$ is a positive constant independent of $N$ and $t$.
Integrating the above inequality from 0 to $t$, we have
\begin{equation*}
\begin{aligned}
\|\partial_t u_N(t)\|_{L^2(\mathbb{R}^d)}^2 + \|\gamma\nabla u_N(t)\|_{L^2(\mathbb{R}^d)}^2  + C(\alpha, \beta, \varepsilon)\int_0^t \|\partial_t u_N(\tau)\|_{H^1(\mathbb{R}^d)}^2 d\tau &\\
\le \|w_0\|_{L^2(\mathbb{R}^d)}^2 + \|\gamma u_0\|_{H^1(\mathbb{R}^d)}^2
+ \frac{1}{4\varepsilon} \|f\|_{L^2(0,T;H^{-1}(\mathbb{R}^d))}^2. &
\end{aligned}
\end{equation*}
Here we use the estimate
\begin{equation*}
    \|\hat{\bm{\Pi}}_N u_0\|_{H^1(\mathbb{R}^d)} \le \| u_0\|_{H^1(\mathbb{R}^d)},
    \quad
     \|\hat{\bm{\Pi}}_N w_0\|_{L^2(\mathbb{R}^d)} \le \| w_0\|_{L^2(\mathbb{R}^d)}.
\end{equation*}
Therefore, in view of the third condition in \eqref{eq:abc:cond}, we have
\begin{equation}\label{eq:stab1}
\begin{aligned}
\|\partial_t u_N\|_{L^{\infty}(0,T;L^2(\mathbb{R}^d))} + \|\nabla u_N\|_{L^\infty(0,T; L^2(\mathbb{R}^d))}+ \|\partial_t u_N\|_{L^2(0,T; H^1(\mathbb{R}^d))}
\le C_1(u_0, w_0, f),
\end{aligned}
\end{equation}
where $C_1(u_0, w_0, f)$ is a constant depending on $\|f\|_{L^2(0,T; H^{-1}(\mathbb{R}^d))}$ and $\|u_0\|_{H^1(\mathbb{R}^d)},\; \|w_0\|_{L^2(\mathbb{R}^d)}$.

Here we do not have the Poincar\'{e} inequality for the space direction, so we need to further give the estimate for $\|u_N\|_{L^\infty(0,T;L^2(\mathbb{R}^d))}$.
By taking $v = u_N(t)$ in \eqref{eq:sg} and using the Cauchy-Schwartz and Young inequalities, we have
\begin{equation*}
\begin{aligned}
\frac{1}{2}\frac{d}{dt}\left(\|\alpha^{1/2} u_N(t)\|_{L^2(\mathbb{R}^d)}^2 + \|\beta^{1/2} \nabla u_N(t)\|_{L^2(\mathbb{R}^d)}^2 \right) + \|\gamma\nabla u_N(t)\|_{L^2(\mathbb{R}^d)}^2 = \left ( f-{\partial_t^2 u_N}(t), u_N(t)\right ) & \\
\le \|f\|_{H^{-1}(\mathbb{R}^d)} \|u_N(t)\|_{H^1(\mathbb{R}^d)} - \left ( {\partial_t^2 u_N}(t), u_N(t)\right ) &\\
\le\frac{1}{2}\big(\|u_N(t)\|_{H^1(\mathbb{R}^d)}^2+\|f\|_{H^{-1}(\mathbb{R}^d)}^2\big) -\left ( {\partial_t^2 u_N}(t), u_N(t)\right ). &
\end{aligned}
\end{equation*}
Integrating the above inequality from 0 to $t$ and using integral by parts with respect to $\tau$ for $\int_0^t\left({\partial_\tau^2 u_N}(\tau), u_N(\tau)\right)d\tau$, we have
\begin{equation*}
\begin{aligned}
&\|\alpha^{1/2} u_N(t)\|_{L^2(\mathbb{R}^d)}^2 + \|\beta^{1/2} \nabla u_N(t)\|_{L^2(\mathbb{R}^d)}^2 + 2\int_0^t \|\gamma\nabla u_N(\tau)\|_{L^2(\mathbb{R}^d)}^2 d\tau \\
\le & \|u_N(t)\|_{L^{2}(0,T;H^1(\mathbb{R}^d))}^2 + \|f\|_{L^{2}(0,T;H^{-1}(\mathbb{R}^d))}^2 + 2\|\partial_t u_N(t)\|_{L^2(0,T;L^2(\mathbb{R}^d))}^2 + {2\int_{\Omega}\left(u_0w_0 - u_N(t)\partial_t u_N(t) \right)d{\bm x}}  \\
\le & \|u_N(t)\|_{L^{2}(0,T;H^1(\mathbb{R}^d))}^2 + \|f\|_{L^{2}(0,T;H^{-1}(\mathbb{R}^d))}^2 + 2\|\partial_t u_N(t)\|_{L^2(0,T;L^2(\mathbb{R}^d))}^2 + 2\|u_0\|_{L^2(\mathbb{R}^d)} \|w_0\|_{L^2(\mathbb{R}^d)} \\
& +\epsilon\|u_N(t)\|_{L^2(\mathbb{R}^d)}^2 + \frac{1}{\epsilon}\|\partial_t u_N(t)\|_{L^2(\mathbb{R}^d)}^2,
\end{aligned}
\end{equation*}
where the Cauchy-Schwarz and Young inequalities are used for the last inequality.
Let $\epsilon = \min(\alpha)/2$, then the following estimate
%
\begin{equation}\label{eq:stab2}
\begin{aligned}
\|u_N\|_{L^\infty(0,T;L^2(\mathbb{R}^d))}^2 + \|\nabla u_N\|_{L^\infty(0,T;L^2(\mathbb{R}^d))}^2 +  \|\nabla u_N\|_{L^2(0,T;L^2(\mathbb{R}^d))}^2
\le C_2(u_0, w_0, f)
\end{aligned}
\end{equation}
follows based on the estimate \eqref{eq:stab1},
where $C_2(u_0, w_0, f)$ is a constant depending on $\|f\|_{L^2(0,T;H^{-1}(\mathbb{R}^d))}$ and $\|u_0\|_{H^1(\mathbb{R}^d)},\; \|w_0\|_{L^2(\mathbb{R}^d)}$.
Therefore, the estimate \eqref{eq:stab3} follows from \eqref{eq:stab1} and \eqref{eq:stab2}.
Furthermore,
we have from \eqref{eq:sg} that
\begin{equation*}
\begin{aligned}
\|\partial_t^2 u_N(t)\|_{H^{-1}(\mathbb{R}^d)}
\le &\|\alpha^{1/2} \partial_t u_N\|_{H^{-1}(\mathbb{R}^d)} + \|\beta^{1/2} \partial_t \nabla u_N\|_{H^{-1}(\mathbb{R}^d)} + \|\gamma \nabla u_N\|_{L^2(\mathbb{R}^d)} + \|f\|_{H^{-1}(\mathbb{R}^d)}\\
\le & \|\alpha^{1/2} \partial_t u_N\|_{L^2(\mathbb{R}^d)} + \|\beta^{1/2} \partial_t \nabla u_N\|_{L^2(\mathbb{R}^d)} + \|\gamma \nabla u_N\|_{L^2(\mathbb{R}^d)} + \|f\|_{H^{-1}(\mathbb{R}^d)}.
\end{aligned}
\end{equation*}
Then the estimate \eqref{eq:stab4} holds by integrating the above equation from 0 to $t$  and using the estimate \eqref{eq:stab3}.
\end{proof}

We are now able to show the existence and uniqueness of the continuous problem \eqref{eq:wk}. We use more or less the standard compactness arguments similar as that used in \cite{han} based on the properties of the Hermite spectral Galerkin approximation.

We state the {main result} of this section concerning the \emph{well-posedness} of the weak problem \eqref{eq:wk} as follows:
\begin{theorem}
Assume $u_0\in H^1(\mathbb{R}^d),\; w_0\in L^2(\mathbb{R}^d),\; f\in L^2(0,T;H^{-1}(\mathbb{R}^d))$, then the weak problem \eqref{eq:wk} admits a unique solution $u\in L^{\infty}(0,T; H^1(\mathbb{R}^d))$ and $\partial_t u\in L^{\infty}(0,T;L^2(\mathbb{R}^d)) \cap L^{2}(0,T;H^1(\mathbb{R}^d)),\partial_{t}^2 u\in L^{2}(0,T; H^{-1}(\mathbb{R}^d))$ satisfying
\begin{equation}\label{eq:conti:stab3}
\begin{aligned}
\|\partial_t u\|_{L^{\infty}(0,T;L^2(\mathbb{R}^d))} + \|\partial_t u\|_{L^2(0,T; H^1(\mathbb{R}^d))} + \|u\|_{L^\infty(0,T; H^1(\mathbb{R}^d))}
\le {C}(u_0, w_0, f).
\end{aligned}
\end{equation}
\end{theorem}

\begin{proof}
We first show the \emph{existence} of the weak problem \eqref{eq:wk}. We have from Theorem \ref{thm:SG:est} that there exists a sequence of functions $\{u_N:\, N \in \mathbb{N} \}$ with  $u_N\in L^\infty(0,T; H^1(\mathbb{R}^d))$, and $\partial_t u_N \in L^\infty(0,T; L^2(\mathbb{R}^d))\cap L^2(0,T; H^1(\mathbb{R}^d))$, and $\partial_t^2 u_N \in L^2(0,T; H^{-1}(\mathbb{R}^d))$ for all $T\ge 0$, satisfying \eqref{eq:sg}-\eqref{eq:stab4} for each $N\ge 1$, and
$$u_N(0) = \hat{\bm{\Pi}}_N u_0, \quad \partial_t u_N(0) = \hat{\bm{\Pi}}_N w_0.$$

By virtue of estimates \eqref{eq:stab3} and \eqref{eq:stab4} and thanks to Theorem 3 of Appendix D.4 in \cite{Evans}, we have that there exists a function $u \in L^\infty(0,T; H^1(\mathbb{R}^d))$, and $\partial_t u \in L^\infty(0,T; L^2(\mathbb{R}^d))\cap L^2(0,T; H^1(\mathbb{R}^d))$, and $\partial_t^2 u \in L^2(0,T; H^{-1}(\mathbb{R}^d))$ such that
\begin{align}
    u_N &\rightharpoonup^* u \text{ in } L^\infty(0,T; H^1(\mathbb{R}^d)),  \label{eq:weakcon1}\\
    u_N &\rightharpoonup  u \text{ in } L^2(0,T; H^1(\mathbb{R}^d)), \label{eq:weakcon2}\\
    \partial_t u_N &\rightharpoonup^* \partial_tu \text{ in } L^\infty(0,T; L^2(\mathbb{R}^d)),  \label{eq:weakcon3}\\
    \partial_t u_N &\rightharpoonup  \partial_t u \text{ in } L^2(0,T; H^1(\mathbb{R}^d)), \label{eq:weakcon4} \\
    \partial_t^2 u_N &\rightharpoonup^* \partial_t^2 u \text{ in } L^2(0,T; H^{-1}(\mathbb{R}^d)) \label{eq:weakcon5}
\end{align}
for all $T\ge 0$ as $N\rightarrow \infty$.

Let $v(x)\in H^1(\mathbb{R}^d)$ and $\eta(t) \in C([0,T])$, we take $\hat{\bm{\Pi}}_Nv$ as the test function in \eqref{eq:sg} and multiply both sides of the resulting identity by $\eta(t)$ and integrate over $t\in [0,T]$ to get
\begin{equation}\label{eq:sg:st}
\begin{aligned}
\int_0^T \left ( {\partial_t^2 u_N}(t), \eta(t)\hat{\bm{\Pi}}_N v\right ) dt +
\int_0^T (\alpha {\partial_t u_N}(t), \eta(t)\hat{\bm{\Pi}}_N v) dt +
\int_0^T (\beta \partial_t \nabla u_N(t), \eta(t)\nabla \hat{\bm{\Pi}}_N v) dt &\\
+ \int_0^T (\gamma^2 \nabla u_N(t), \eta(t)\nabla \hat{\bm{\Pi}}_N v) dt
= \int_0^T (f,\eta(t)\hat{\bm{\Pi}}_Nv) dt.
\end{aligned}
\end{equation}
By taking the limit $N\rightarrow \infty$ and applying the convergence results \eqref{eq:weakcon1}-\eqref{eq:weakcon5}, we obtain
\begin{align*}
    &\int_0^T \left ( {\partial_t^2 u_N}(t), \eta(t)\hat{\bm{\Pi}}_N v\right ) dt \rightarrow
    \int_0^T \left ( {\partial_t^2 u}(t), \eta(t) v\right ) dt, \\
    &\int_0^T (\alpha {\partial_t u_N}(t), \eta(t)\hat{\bm{\Pi}}_N v) dt \rightarrow
    \int_0^T (\alpha {\partial_t u}(t), \eta(t)v) dt, \\
    &\int_0^T (\beta \partial_t \nabla u_N(t), \eta(t)\nabla \hat{\bm{\Pi}}_N v) dt \rightarrow
    \int_0^T (\beta \partial_t \nabla u(t), \eta(t)\nabla  v) dt, \\
    &\int_0^T (\gamma^2 \nabla u_N(t), \eta(t)\nabla \hat{\bm{\Pi}}_N v) dt \rightarrow
    \int_0^T (\gamma^2 \nabla u(t), \eta(t)\nabla v) dt,
\end{align*}
and
\begin{align*}
    \int_0^T (f,\eta(t)\hat{\bm{\Pi}}_N v) dt \rightarrow
    \int_0^T (f,\eta(t)v) dt.
\end{align*}
Consequently, we get by letting $N\rightarrow \infty$ in \eqref{eq:sg:st} that
\begin{equation}\label{eq:wk:st}
\begin{aligned}
\int_0^T \left ( {\partial_t^2 u}(t), \eta(t) v\right ) dt +
\int_0^T (\alpha {\partial_t u}(t), \eta(t) v) dt +
\int_0^T (\beta \partial_t \nabla u(t), \eta(t)\nabla v) dt &\\
+ \int_0^T (\gamma^2 \nabla u(t), \eta(t)\nabla v) dt
= \int_0^T (f,\eta(t) v) dt,~~ \quad \forall v\in H^1(\mathbb{R}^d). &
\end{aligned}
\end{equation}
Since $\eta(t)\in C([0,T])$ is arbitrary, we deduce that $u(t)$ satisfies \eqref{eq:wk} for all $T > 0$.

We now show that $u(0) = u_0$ and $\partial_t u(0) = w_0$ to complete the proof of the existence of the weak solution of \eqref{eq:wk}. It follows from \eqref{eq:weakcon2} and \eqref{eq:weakcon4} that
\begin{align*}
    u(0) = \lim_{N\rightarrow \infty} \hat{\bm{\Pi}}_N u_0 = u_0.
\end{align*}
Note that $\partial_t u(t)\in L^2(0,T; H^1(\mathbb{R}^d))$ and $\partial_t^2 u(t)\in L^2(0,T;H^{-1}(\mathbb{R}^d))$, then we have
\begin{align*}
    \|\partial_t u(t)\|_{L^2(\mathbb{R}^d)}^2 = 2\int_0^t (\partial_\tau^2 u(\tau), \partial_\tau u(\tau)) d\tau + \|\partial_t u(0)\|_{L^2(\mathbb{R}^d)}^2,
\end{align*}
which implies that $\partial_t u(t) \in C([0,T]; L^2(\mathbb{R}^d))$. Therefore, by replacing $\eta(t)$ with $\eta_T(t) = 1-t/T$ and integrating by parts against $t$ for the first term of  \eqref{eq:wk:st}, we obtain
\begin{equation}\label{eq:wk:st1}
\begin{aligned}
\int_0^T \frac{1}{T} \left ( {\partial_t u}(t), v\right ) dt +
\int_0^T (\alpha {\partial_t u}(t), \eta_T(t) v) dt +
\int_0^T (\beta \partial_t \nabla u(t), \eta_T(t)\nabla v) dt &\\
+ \int_0^T (\gamma^2 \nabla u(t), \eta_T(t)\nabla v) dt
= \int_0^T (f,\eta_T(t) v) dt
+ (\partial_t u(0), v). &
\end{aligned}
\end{equation}
Then, by applying the same argument for \eqref{eq:sg:st} with $\partial_t u_N(t) = \hat{\bm{\Pi}}_N w_0$, we obtain
\begin{equation*}
\begin{aligned}
\int_0^T \frac{1}{T} \left ( {\partial_t u_N}(t), \hat{\bm{\Pi}}_N v\right ) dt +
\int_0^T (\alpha {\partial_t u_N}(t), \eta_T(t)\hat{\bm{\Pi}}_N v) dt +
\int_0^T (\beta \partial_t \nabla u_N(t), \eta_T(t)\nabla \hat{\bm{\Pi}}_N v) dt &\\
+ \int_0^T (\gamma^2 \nabla u_N(t), \eta_T(t)\nabla \hat{\bm{\Pi}}_N v) dt
= \int_0^T (f,\eta_T(t)\hat{\bm{\Pi}}_Nv)
+ (\hat{\bm{\Pi}}_N w_0, \hat{\bm{\Pi}}_N v). &
\end{aligned}
\end{equation*}
Taking $N\rightarrow \infty$ gives
\begin{equation*}
\begin{aligned}
\int_0^T \frac{1}{T} \left ( {\partial_t u}(t), v\right ) dt +
\int_0^T (\alpha {\partial_t u}(t), \eta_T(t) v) dt +
\int_0^T (\beta \partial_t \nabla u(t), \eta_T(t)\nabla v) dt &\\
+ \int_0^T (\gamma^2 \nabla u(t), \eta_T(t)\nabla v) dt
= \int_0^T (f,\eta_T(t) v) dt
+ (w_0, v). &
\end{aligned}
\end{equation*}
Comparing the above equation with \eqref{eq:wk:st1}, we have $\partial_t u(0) = w_0$.
The estimate \eqref{eq:conti:stab3} follows by letting $N\rightarrow \infty$ in \eqref{eq:stab3}.

We now show the \emph{uniqueness} of the weak solution.
Let $u$ and $\bar{u}$ be two solutions of the weak problem \eqref{eq:wk} with $u(0) = \bar{u}(0) = u_0$ and $\partial_t u(0) = \partial_t \bar{u}(0) = w_0$. Denote $e = u-\bar{u}$. Then for a.e. $t\in (0,T)$, $e$ satisfies
\begin{equation*}
\left ( {\partial_t^2 e}(t), v\right ) +(\alpha {\partial_t e}(t), v) + (\beta \partial_t \nabla e(t), \nabla v) + (\gamma^2 \nabla e(t), \nabla v) = 0,\quad \forall v\in H^1(\mathbb{R}^d)
\end{equation*}
with $e(0) = \partial_t e(0) = 0$.
Taking $v = \partial_t e(t)$ in the above equation, we obtain
\begin{equation*}
\begin{aligned}
\frac{1}{2}\frac{d}{dt}\left(\|\partial_t e(t)\|_{L^2(\mathbb{R}^d)}^2 + \|\gamma\nabla e(t)\|_{L^2(\mathbb{R}^d)}^2 \right) +\|\alpha^{1/2}\partial_t e(t)\|_{L^2(\mathbb{R}^d)}^2 + \|\beta^{1/2}\partial_t\nabla e(t)\|_{L^2(\mathbb{R}^d)}^2
= 0.
\end{aligned}
\end{equation*}
This yields
\begin{equation*}
\begin{aligned}
\frac{1}{2}\frac{d}{dt}\left(\|\partial_t e(t)\|_{L^2(\mathbb{R}^d)}^2 + \|\gamma\nabla e(t)\|_{L^2(\mathbb{R}^d)}^2 \right)
\le 0.
\end{aligned}
\end{equation*}
Integrating the above equation from 0 to $t$, and noting that $e(0) = \partial_t e(0) = 0$, we have
\begin{equation*}
\begin{aligned}
\|\partial_t e(t)\|_{L^2(\mathbb{R}^d)}^2 + \|\gamma\nabla e(t)\|_{L^2(\mathbb{R}^d)}^2  \le 0.
\end{aligned}
\end{equation*}
Hence, $\partial_t e (t) = 0$. Using the initial condition $e(0) = 0$ again, we have $e(t) = 0$ for all $t\in [0,T]$, i.e., the weak problem \eqref{eq:wk} has a unique solution.
\end{proof}

\begin{remark}
By using the same argument for the uniqueness, we can readily show the stability of the solution in terms of the initial conditions $u_0, \; w_0$ and the source term $f$.
\end{remark}

\subsection{Regularity}\label{subsec:regulaity}
We now discuss how the regularity of the solution of the DVWE depends on the initial conditions and the source term $f$ with sufficiently smooth coefficients $\alpha, \beta,\gamma$. For the sake of simplicity, we assume that $\alpha,~ \beta,~ \gamma $ are constants.

Let $\tilde{u}_N:= \partial_t u_N$, and  $\tilde{u}_0 = w_0,\; \tilde{w}_0 = \tilde{w}_0^1+ \tilde{w}_0^2$ with $\tilde{w}_0^1= f(\cdot, 0)- {\rm div}(\gamma^2 \nabla u_0),\; \tilde{w}_0^2 = \alpha w_0-{\rm div}(\beta\nabla w_0)$, then we have
\begin{equation}\label{eq:sg:ut}
A(\tilde{u}_N,v) = (f_t,v), \quad \forall v\in V_N
\end{equation}
with $\tilde{u}_N(\bm{x},0)=\hat{\bm{\Pi}}_N \tilde{u}_0,\;\partial_t \tilde{u}_N(\bm{x},0)=\hat{\bm{\Pi}}_N \tilde{w}_0$.
By using the similar arguments in the last section, we obtain the following energy estimate:
\begin{equation*}
\begin{aligned}
&\|\partial_t \tilde{u}_N\|_{L^{\infty}(0,T;L^2(\mathbb{R}^d))} + \|\partial_t \tilde{u}_N\|_{L^2(0,T;H^1(\mathbb{R}^d))} + \|\tilde{u}_N\|_{L^\infty(0,T;H^1(\mathbb{R}^d))}\\
\le &{C}(\|f_t\|_{L^2(0,T; H^{-1}(\mathbb{R}^d))}+ \|\tilde{u}_0\|_{H^1(\mathbb{R}^d)}+ \|\tilde{w}_0^1\|_{L^2(\mathbb{R}^d)}+ \|\tilde{w}_0^2\|_{L^2(\mathbb{R}^d)})\\
\le & {C}(\|f_t\|_{L^2(0,T; H^{-1}(\mathbb{R}^d))}+ \|f(\cdot,0)\|_{L^2(\mathbb{R}^d)}+ \|{u}_0\|_{H^2(\mathbb{R}^d)}+ \|{w}_0\|_{H^2(\mathbb{R}^d)}).
\end{aligned}
\end{equation*}
This gives
\begin{equation}\label{eq:stab3:ut}
\begin{aligned}
&\|\partial_t^2 {u}_N\|_{L^{\infty}(0,T;L^2(\mathbb{R}^d))} + \|\partial_t^2 {u}_N\|_{L^2(0,T;H^1(\mathbb{R}^d))} + \|\partial_t {u}_N\|_{L^\infty(0,T;H^1(\mathbb{R}^d))}\\
\le & {C}(\|f_t\|_{L^2(0,T; H^{-1}(\mathbb{R}^d))}+ \|f(\cdot,0)\|_{L^2(\mathbb{R}^d)}+ \|{u}_0\|_{H^2(\mathbb{R}^d)}+ \|{w}_0\|_{H^2(\mathbb{R}^d)}).
\end{aligned}
\end{equation}
Take $v = -\Delta u_N$ in \eqref{eq:sg}, we obtain
\begin{equation*}\label{eq:eng:B}
\begin{aligned}
\frac{1}{2}\frac{d}{dt} \left (\|\beta^{1/2}\Delta u_N\|_{L^2(\mathbb{R}^d)}^2 + \|\alpha^{1/2}\nabla u_N\|_{L^2(\mathbb{R}^d)}^2\right) + \|\gamma \Delta u_N\|_{L^2(\mathbb{R}^d)}^2 = (f-\partial^2_{t}u_N, -\Delta u_N) &\\
\le \frac{1}{4\gamma^2}\|f-\partial^2_{t}u_N\|_{L^2(\mathbb{R}^d)}^2 + \|\gamma \Delta u_N\|_{L^2(\mathbb{R}^d)}^2. &
\end{aligned}
\end{equation*}
Integrating the above equation from 0 to $t$ and using the estimate \eqref{eq:stab3:ut}, we obtain
\begin{equation*}
\begin{aligned}
\|\Delta u_N\|_{L^2(\mathbb{R}^d)}^2 + \|\nabla u_N\|_{L^2(\mathbb{R}^d)}^2
\le C(\|f\|_{H^1(0,T; L^2(\mathbb{R}^d))}^2 + \|u_0\|_{H^2(\mathbb{R}^d)}^2 + \|w_0\|_{H^2(\mathbb{R}^d)}^2).
\end{aligned}
\end{equation*}
By letting $N\rightarrow \infty$, we obtain from the above estimate and \eqref{eq:stab3:ut} that
\begin{equation}\label{eq:reg:k=0}
\begin{aligned}
&\|\partial_t^2 {u}\|_{L^{\infty}(0,T;L^2(\mathbb{R}^d))} + \|\partial_t {u}\|_{L^\infty(0,T;H^1(\mathbb{R}^d))} + \|u\|_{L^\infty(0,T;H^2(\mathbb{R}^d))} + \|\partial_t^2 {u}\|_{L^2(0,T;H^1(\mathbb{R}^d))}\\
\le & C(\|f\|_{H^1(0,T; L^2(\mathbb{R}^d))} +  \|u_0\|_{H^2(\mathbb{R}^d)} + \|w_0\|_{H^2(\mathbb{R}^d)}).
\end{aligned}
\end{equation}




We next show the higher regularity result. In particular, we have the following result:
\begin{theorem}
Assume $\alpha, \beta, \gamma$ are sufficiently smooth, let $u$ be the solution of  \eqref{eq:wk}, if $u_0\in H^{k+2}(\mathbb{R}^d),\; w_0\in H^{k+2}(\mathbb{R}^d),\; f\in H^1(0,T;H^{k}(\mathbb{R}^d)),k\ge 0$, then
\begin{equation}\label{eq:reg:kg0}
\begin{aligned}
&\|\partial_t^2 u\|_{L^\infty(0,T;H^k(\mathbb{R}^d))}
+\|\partial_t u\|_{L^\infty(0,T;H^{k+1}(\mathbb{R}^d))}
+\|u\|_{L^\infty(0,T;H^{k+2}(\mathbb{R}^d))}
+\|\partial_t^2 u\|_{L^2(0,T;H^{k+1}(\mathbb{R}^d))}\\
& +\|\partial_t u\|_{L^2(0,T;H^{k+2}(\mathbb{R}^d))} \le C  \left(\|f\|_{H^1(0,T;H^{k}(\mathbb{R}^d))} + \|w_0\|_{H^{k+2}(\mathbb{R}^d)} + \|u_0\|_{H^{k+2}(\mathbb{R}^d)}
\right).
\end{aligned}
\end{equation}
\end{theorem}
\begin{proof}
Taking $v=(-1)^k\partial_t\nabla^{2k}\tilde{u}_N, ~k\ge 1$ in \eqref{eq:sg:ut}, we obtain
\begin{equation*}
\begin{aligned}
&\frac{1}{2}\frac{d}{dt}\left(\|\partial_t\nabla^{k}\tilde{u}_N\|_{L^2(\mathbb{R}^d)}^2+\|\gamma\nabla^{k+1}\tilde{u}_N\|_{L^2(\mathbb{R}^d)}^2\right)
+\alpha\|\partial_t\nabla^{k}\tilde{u}_N\|_{L^2(\mathbb{R}^d)}^2+\beta\|\partial_t\nabla^{k+1}\tilde{u}_N\|_{L^2(\mathbb{R}^d)}^2\\
&=(f_t,\partial_t\nabla^{2k}u_N)\le C(\beta)\|f_t\|_{H^{k-1}(\mathbb{R}^d)}^2
+\frac{\beta}{2}\|\partial_t\nabla^{k+1}u_N\|_{L^2(\mathbb{R}^d)}^2.
\end{aligned}
\end{equation*}
Integrating the above equation from 0 to $T$, we have
\begin{equation*}
\begin{aligned}
&\|\partial_t\nabla^{k}\tilde{u}_N\|_{L^\infty(0,T;L^2(\mathbb{R}^d))}^2 + \|\gamma\nabla^{k+1}\tilde{u}_N\|_{L^\infty(0,T;L^2(\mathbb{R}^d))}^2 + 2\alpha\|\partial_t\nabla^{k}\tilde{u}_N\|_{L^2(0,T;L^2(\mathbb{R}^d))}^2 \\
&+{\beta}\|\partial_t\nabla^{k+1}\tilde{u}_N\|_{L^2(0,T;L^2(\mathbb{R}^d))}^2 \le C\left(\|f_t\|_{L^2(0,T;H^{k-1}(\mathbb{R}^d))}^2+\|\nabla^{k+1} \tilde{u}_0\|_{L^2(\mathbb{R}^d)}^2+\|\nabla^k \tilde{w}_0\|_{L^2(\mathbb{R}^d)}^2\right),
\end{aligned}
\end{equation*}
which leads to
\begin{equation}\label{eq:reg:pf1}
\begin{aligned}
&\|\partial_t^2\nabla^{k}u_N\|_{L^\infty(0,T;L^2(\mathbb{R}^d))}^2
+\|\partial_t\nabla^{k+1}u_N\|_{L^\infty(0,T;L^2(\mathbb{R}^d))}^2
+\|\partial_t^2\nabla^{k+1}u_N\|_{L^2(0,T;L^2(\mathbb{R}^d))}^2\\
&\le C\left(\|f_t\|_{L^2(0,T;H^{k-1}(\mathbb{R}^d))}^2 + \|f(\cdot,0)\|_{H^{k}(\mathbb{R}^d)}^2 +\|\nabla^{k+2} w_0\|_{L^2(\mathbb{R}^d)}^2 + \|\nabla^{k+2} u_0\|_{L^2(\mathbb{R}^d)}^2
\right).
\end{aligned}
\end{equation}
Using the same argument for $\tilde{u}_N$ to $u_N$, we obtain the following estimate
\begin{equation}\label{eq:reg:pf2}
\begin{aligned}
&\|\partial_t\nabla^{k}{u}_N\|_{L^\infty(0,T;L^2(\mathbb{R}^d))}^2 + \|\nabla^{k+1}{u}_N\|_{L^\infty(0,T;L^2(\mathbb{R}^d))}^2 +\|\partial_t\nabla^{k+1}{u}_N\|_{L^2(0,T;L^2(\mathbb{R}^d))}^2 \\
& \le C\left(\|f\|_{L^2(0,T;H^{k-1}(\mathbb{R}^d))}^2+\|\nabla^{k+1} {u}_0\|_{L^2(\mathbb{R}^d)}^2+\|\nabla^k {w}_0\|_{L^2(\mathbb{R}^d)}^2\right).
\end{aligned}
\end{equation}
%
%
We put the above two estimates together as follow:
\begin{equation*}
\begin{aligned}
&\|\partial_t^2 u_N\|_{L^\infty(0,T;H^k(\mathbb{R}^d))}
+\|\partial_t u_N\|_{L^\infty(0,T;H^{k+1}(\mathbb{R}^d))}
+\|u_N\|_{L^\infty(0,T;H^{k+2}(\mathbb{R}^d))}
+\|\partial_t^2 u_N\|_{L^2(0,T;H^{k+1}(\mathbb{R}^d))}\\
& +\|\partial_t u_N\|_{L^2(0,T;H^{k+2}(\mathbb{R}^d))} \le C \left(\|f\|_{H^1(0,T;H^{k}(\mathbb{R}^d))} + \|w_0\|_{H^{k+2}(\mathbb{R}^d)} + \| u_0\|_{H^{k+2}(\mathbb{R}^d)}
\right).
\end{aligned}
\end{equation*}
The estimate \eqref{eq:reg:kg0} follows by letting $N\to\infty$.
\end{proof}

\section{Error estimates}\label{sec:err}
We show in this section the convergence of the Hermite spectral Galerkin approximation.

Let us denote the errors as follows:
\begin{equation*}\label{eq21}
e_N=u-u_N=\xi_N+\eta_N,~~\xi_N=\widehat{\bm\Pi}_Nu-u_N,~~~\eta_N=u-\widehat{\bm\Pi}_Nu.
\end{equation*}
We have the following error estimate for the Hermite spectral Galerkin approximation.

\begin{theorem}
Let $u$ and $u_N$ be the solutions of the weak problem \eqref{eq:wk} and  the Hermite spectral Galerkin problem \eqref{eq:sg}, respectively. Assume $u_0 \in \hat{B}_{\mu}(\mathbb{R}^d),\; w_0 \in \hat{B}_{\nu}(\mathbb{R}^d),\; u(t)\in L^2(0,T;\hat{B}_r(\mathbb{R}^d))\cap L^\infty(0,T;\hat{B}_s(\mathbb{R}^d)),\; \partial_t u(t) \in L^2(0,T;\hat{B}_q(\mathbb{R}^d))\cap L^\infty(0,T;\hat{B}_{\tau}(\mathbb{R}^d)), \; \partial_t^2 u(t) \in L^2(0,T;\hat{B}_p(\mathbb{R}^d)),t\ge 0, \; q,s>1, p,\mu,\nu,\tau>0$, then there holds the following error estimate:
\begin{equation}\label{eq:errest1}
\begin{aligned}
&\|\partial_te_N\| + \|e_N \|_{H^1(\mathbb{R}^d)}
\lesssim N^{-\frac{p}{2}} |\partial_t^2u|_{L^2(0,T;\hat{B}_p(\mathbb{R}^d))} +
N^{\frac{1-q}{2}}|\partial_tu|_{L^2(0,T;\hat{B}_q(\mathbb{R}^d))}
+ N^{-\frac{\tau}{2}}|\partial_tu|_{L^\infty(0,T;\hat{B}_{\tau}(\mathbb{R}^d))}\\
&+ N^{\frac{1-r}{2}}|u|_{L^2(0,T;\hat{B}_r(\mathbb{R}^d))}  + N^{\frac{1-s}{2}}|u|_{L^\infty(0,T;\hat{B}_s(\mathbb{R}^d))}
+N^{-\frac{\mu}{2}}|u_0|_{\hat{B}_{\mu}(\mathbb{R}^d)} +
N^{-\frac{\nu}{2}}|w_0|_{\hat{B}_{\nu}(\mathbb{R}^d)}
\end{aligned}
\end{equation}
for all $t>0$. Here and below $\|\cdot\|$ denotes for the standard $L^2$ norm.
\end{theorem}
\begin{proof}
We derive from \eqref{eq:wk} and \eqref{eq:sg} that
\begin{equation*}
A(u-u_N,v)=0,\quad~\forall v\in V_N,
\end{equation*}
where $A(\cdot,\cdot)$ is defined in \eqref{eq22}.
Consequently, we have
\begin{equation}\label{eq28}
A(\xi_N,v)=-A(\eta_N,v), \quad~~\forall v\in V_N.
\end{equation}
By taking $v=\partial_t\xi_N$ in \eqref{eq28}, we get
\begin{equation}\label{eq26}
A\big(\xi_N,\partial_t\xi_N\big)=\frac{1}{2}\frac{d}{dt}\|\partial_t\xi_N\|^2+\|\alpha^{\frac{1}{2}}\partial_t\xi_N\|^2+\|\beta^{\frac{1}{2}}\partial_t\nabla\xi_N\|^2+
\frac{1}{2}\frac{d}{dt}\|\gamma\nabla\xi_N\|^2,
\end{equation}
and
\begin{equation}\label{eq27}
\begin{aligned}
A\big(\eta_N,\partial_t\xi_N\big)&=\big(\partial_t^2\eta_N,\partial_t\xi_N\big)+\big(\alpha\partial_t\eta_N,\partial_t\xi_N\big)
+\big(\beta\partial_t\nabla\eta_N,\partial_t\nabla\xi_N\big)+\big(\gamma^2\nabla\eta_N,\partial_t\nabla\xi_N\big)\\
&\le \epsilon_1\|\partial_t\xi_N\|^2 + C(\epsilon_1)\|\partial_t^2\eta_N\|^2
+\epsilon_2 \|\partial_t\xi_N\|^2 + C(\epsilon_2)\|\partial_t\eta_N\|^2 +\epsilon_3\|\partial_t\nabla\xi_N\|^2\\
&~~~+ C(\epsilon_3)\|\partial_t\nabla\eta_N\|^2 + \epsilon_4 \|\partial_t\nabla\xi_N\|^2+C(\epsilon_4)\|\nabla\eta_N\|^2.
\end{aligned}
\end{equation}
Set $\epsilon_1 + \epsilon_2 = \alpha,\; \epsilon_3 + \epsilon_4 = \beta$, then
by combining the above two equations, and using estimate \eqref{eqnErrEstiHerMultiH}, we obtain
\begin{equation*}
\begin{aligned}
\frac{d}{dt}\big(\|\partial_t\xi_N\|^2+\|\gamma\nabla\xi_N\|^2\big)
\le &C(\|\partial_t^2\eta_N\|^2+\|\partial_t\eta_N\|^2+\|\partial_t\nabla\eta_N\|^2+\|\nabla\eta_N\|^2)\\
\le & C(N^{-p} |\partial_t^2u|^2_{\hat{B}_p(\mathbb{R}^d)} +  N^{1-q}|\partial_tu|^2_{\hat{B}_q(\mathbb{R}^d)} +N^{1-r}|u|^2_{\hat{B}_r(\mathbb{R}^d)}),
\end{aligned}
\end{equation*}
where $C$ is a constant independent of $N$. Integrating the above equation from 0 to $t$, we obtain
\begin{equation*}
\begin{aligned}
\|\partial_t\xi_N\|^2+\|\gamma\nabla\xi_N\|^2\lesssim & N^{-p} |\partial_t^2u|^2_{L^2(0,T;\hat{B}_p(\mathbb{R}^d))} +
N^{1-q}|\partial_tu|^2_{L^2(0,T;\hat{B}_q(\mathbb{R}^d))} +
N^{1-r}|u|^2_{L^2(0,T;\hat{B}_r(\mathbb{R}^d))}\\
&+N^{-\mu}\|u_0\|_{\hat{B}_{\mu}(\mathbb{R}^d)}^2 + N^{-\nu}\|w_0\|_{\hat{B}_{\nu}(\mathbb{R}^d)}^2.
\end{aligned}
\end{equation*}
Therefore, the following estimate
\begin{equation}\label{eq:errpf3}
\begin{aligned}
&\|\partial_te_N \|^2 + \|\nabla e_N \|^2
\lesssim N^{-p} |\partial_t^2u|^2_{L^2(0,T;\hat{B}_p(\mathbb{R}^d))} +
N^{1-q}|\partial_tu|^2_{L^2(0,T;\hat{B}_q(\mathbb{R}^d))} +
N^{1-r}|u|^2_{L^2(0,T;\hat{B}_r(\mathbb{R}^d))} \\
&+N^{-\mu}|u_0|_{\hat{B}_{\mu}(\mathbb{R}^d)}^2 +
N^{-\nu}|w_0|_{\hat{B}_{\nu}(\mathbb{R}^d)}^2 +
N^{-\tau}|\partial_tu|^2_{L^\infty(0,T;\hat{B}_{\tau}(\mathbb{R}^d))}+ N^{1-s}|u|^2_{L^\infty(0,T;\hat{B}_s(\mathbb{R}^d))}.
\end{aligned}
\end{equation}
follows by using the triangle inequality and the projection property. To obtain the estimate \eqref{eq:errest1}, we are left to estimate $\|e_N(t)\|$. Note that
\begin{align*}
    \frac{d}{dt}\|e_N(t)\|^2
    \le  2\|e_N(t)\| \|\partial_t e_N(t)\|
    \le  \|e_N(t)\|^2  + \|\partial_t e_N(t)\|^2.
\end{align*}
Then using the Gronwall's inequality, we obtain
\begin{align*}
    \|e_N(t)\| \le C \|\partial_t e_N(t)\|.
\end{align*}
Consequently, the estimate \eqref{eq:errest1} follows by the estimate \eqref{eq:errpf3}.
\end{proof}

\begin{remark}
If a function $v$ decays sufficiently fast at infinity, then $\|v\|_{H^k(\mathbb{R}^d)} \approx \|v\|_{\hat{B}_k(\mathbb{R}^d)}, k\ge 0$. Therefore, if $u,\partial_t u , \partial_t^2 u$ decay sufficiently fast at infinity, $\alpha, \beta, \gamma$ are sufficiently smooth, and $f\in H^1(0,T;H^{k}(\mathbb{R}^d)),\; u_0 \in H^{k+2}(\mathbb{R}^d),\; w_0 \in H^{k+2}(\mathbb{R}^d)$, in view of the estimates \eqref{eq:reg:kg0}, we observe that the estimate \eqref{eq:errest1} is reduced as
\begin{equation}\label{eq:errest2}
\begin{aligned}
\|\partial_te_N\| + \|e_N \|_{H^1(\mathbb{R}^d)}
\lesssim N^{-\frac{k+1}{2}} \left(\|f\|_{H^1(0,T;H^{k}(\mathbb{R}^d))} + \|w_0\|_{H^{k+2}(\mathbb{R}^d)} + \|u_0\|_{H^{k+2}(\mathbb{R}^d)}
\right).
\end{aligned}
\end{equation}
\end{remark}

\section{Implementation and numerical examples}\label{sec:num}
In this section, we shall briefly give the implementation details, and then present several numerical examples to demonstrate the proposed algorithm.

\subsection{Implementation}
Now let us give the details of the  implementation. We first consider the space discretization. Let
\begin{align*}
    {\Phi}_{|{\bm j}|_1}( {\bm x}) = \prod_{k=1}^d {\phi}_{j_k}(x_k),
\end{align*}
where $|{\bm j}|_1 = j_1 + \ldots +j_d$.
Then $\{\Phi_j( {\bm x})\}_{j=1}^{N_b}$ are the basis functions of $V_N$. We can express the numerical approximation $u_N$ as
\begin{equation}\label{eq7}
u_N(\bm{x},t)=\sum\limits_{j=1}^{N_b}\widehat{u}_j(t)\Phi_j(\bm{x}),
\end{equation}
where $N_b$ is the total number of basis functions.

By setting $v=\Phi_i({\bm x})~(i=1,\cdots,N_b)$ and using the formulation \eqref{eq7}, we can rewrite
\eqref{eq:sg} into the following linear system
\begin{equation}\label{eq8}
M\frac{d^2}{dt^2}\widehat{U}(t)+M_\alpha\frac{d}{dt}\widehat{U}(t)
+S_\beta\frac{d}{dt}\widehat{U}(t)+S_\gamma\widehat{U}(t)=F(t),
\end{equation}
where the mass matrix $M$ and weighted mass matrix $M_\alpha$ are given by
$$M_{ij}=(\Phi_j,\Phi_i),~~M_{\alpha,{ij}}=(\alpha\Phi_j,\Phi_i),$$
and the weighted stiffness matrices $S_\beta$ and $S_\gamma$ are given by
$$S_{\beta,{ij}}=(\beta\nabla\Phi_j,\nabla\Phi_i),~~S_{\gamma,{ij}}
=(\gamma^2\nabla\Phi_j,\nabla\Phi_i),$$
%
%
and
$$F(t) = \left[(f,\Phi_1),\ldots, (f,\Phi_{N_b})\right]^T,\;
\widehat{U}(t)= \left[\widehat{u}_1(t),\ldots, \widehat{u}_{N_b}(t)\right]^T.$$


For the computation of the mass and stiff matrices, the explicit forms can be found in \cite[Section 7.2]{shen} if $\alpha,\beta,\gamma$ are all constants. Otherwise, if $\alpha,\beta,\gamma$ are variable coefficients, to obtain the matrices $M_\alpha,\, S_\beta$ and $S_\gamma$, we use the Gauss-Hermite quadrature. The $n$-point Gauss-Hermite rule in $\mathbb{R}$ reads as
$$\int_{\mathbb{R}}{\rm e}^{-x^2}f(x)dx\approx\sum\limits_{k=1}^n \omega_kf(x_k),$$
where $x_k,\, \omega_k~(k=1,2,\cdots,n)$ are the Hermite Gauss quadrature points and weights, respectively.

Now we introduce the time discretization for solving the resulting linear system \eqref{eq8}.
To get a globally high-order accurate scheme in time,
here we would like to employ the third-stage strong stability
preserving (SSP) Runge-Kutta (RK) method for the ordinary differential system \eqref{eq8}. To this end, an auxiliary variable $\widehat{V}(t)$ is first introduced for \eqref{eq8} to get the following first-order system:
\begin{equation*}
\left\{\begin{aligned}
&\frac{d}{dt}\widehat{U}(t)=\widehat{V}(t),\\
&\frac{d}{dt}\widehat{V}(t)=A\widehat{U}(t)+B\widehat{V}(t)+\widetilde{F}(t),
\end{aligned}\right.
\end{equation*}
where $A=-M^{-1}S_{\gamma},B=-M^{-1}\big(M_\alpha+S_\beta\big)$ and $\widetilde{F}
=M^{-1}F$.

Assume that the time interval is discretized as: $t_{n+1}= t_n+\Delta t_n, n = 0,1,2,\cdots$, where $\Delta t_n$ is the time step size at $t=t_n$.
Let $u^n$ and $v^n$ be
the approximations to $u(\cdot,t^n)$ and $v(\cdot,t_n)$, then the third-stage SSP-RK method is used for the time discretization as follows:
\\
Stage 1:
\begin{equation*}
\begin{aligned}
&\widehat{U}^{(1)}=\widehat{U}^n+\Delta t_n\widehat{V}^n,\\
&\widehat{V}^{(1)}=\widehat{V}^n+\Delta t_n\left(A\widehat{U}^n+B\widehat{V}^n+\widetilde{F}^n\right).
\end{aligned}
\end{equation*}
Stage 2:
\begin{equation*}
\begin{aligned}
&\widehat{U}^{(2)}=\frac{3}{4}\widehat{U}^n+\frac{1}{4}\left(\widehat{U}^{(1)}+\Delta t_n\widehat{V}^{(1)}\right),\\
&\widehat{V}^{(2)}=\frac{3}{4}\widehat{V}^n+\frac{1}{4}\left(\widehat{V}^{(1)}+\Delta t_n\big(A\widehat{U}^{(1)}+B\widehat{V}^{(1)}+\widetilde{F}^{(1)}\big)\right).
\end{aligned}
\end{equation*}
Stage 3:
\begin{equation*}
\begin{aligned}
&\widehat{U}^{n+1}=\frac{1}{3}\widehat{U}^n+\frac{2}{3}\left(\widehat{U}^{(2)}+\Delta t_n\widehat{V}^{(2)}\right),\\
&\widehat{V}^{n+1}=\frac{1}{3}\widehat{V}^n+\frac{2}{3}\left(\widehat{V}^{(2)}+\Delta t_n\big(A\widehat{U}^{(2)}+B\widehat{V}^{(2)}\big)+\widetilde{F}^{(2)}\right).
\end{aligned}
\end{equation*}

\subsection{Numerical examples}
In the following, we present several numerical examples to test the accuracy of the Hermite spectral Galerkin method and to illustrate the behavior of the solutions to problem \eqref{eq1}.

\textbf{Example 1}. Accuracy tests with \emph{smooth} initial conditions and source term in the one-dimensional case.
We consider the 1D diffusive-viscous wave equation with
$\alpha=\beta=\gamma=1$ and different initial conditions and source functions:
\begin{enumerate}[(\romannumeral1)]
\item $f=0,~~~u(x,0)={\rm e}^{-x^2},~~~u_t(x,0)=-{\rm e}^{-x^2}.$
\item $f={\rm e}^{-x^2}[(1-4x^2)\sin t+(3-4x^2)\cos t],~~~u(x,0)=0,~~~u_t(x,0)={\rm e}^{-x^2}.$
\end{enumerate}

These two problems have exact solutions $u(x,t)={\rm e}^{-x^2-t}$ and
$u(x,t)={\rm e}^{-x^2}\sin t$ respectively.
We compute the numerical solutions with the spectral method and the
third-stage SSP Runge-Kutta method ($\Delta t_n=10^{-4}$) until the final time
$T=1$. The degree of the space approximation is $N$. The numerical errors measured by three different norms ($L^2$ and $L^\infty$) and orders of accuracy are listed in Table \ref{tab1}
for case (\romannumeral1) and case (\romannumeral2), respectively. We also show the errors in semi-log scale in Figure \ref{fig1} for the cases (\romannumeral1) and (\romannumeral2). We observe that the expected exponential convergence rates are obtained.
\begin{table}[H]
\centering
\caption{Example 1: Errors and convergence rates for the 1D cases.}
\label{tab1}
\begin{tabular}{c|c|c|c|c||c|c|c|c}
\hline
\hline
 & \multicolumn{4}{c||}{case (\romannumeral1)} & \multicolumn{4}{c}{case (\romannumeral2)}\\
 \hline
$N$ & $L^2$ error & $L^2$ order & $L^\infty$ error
& $L^\infty$ order & $L^2$ error & $L^2$ order & $L^\infty$ error & $L^\infty$ order\\
\hline
  10   &    2.751E-04   &  ---   &    1.316E-04   &  ---&    7.467E-04   &  ---   &    4.783E-04   &  ---\\
  15   &    2.855E-05   &  5.588   &    1.302E-05   &  5.705 &    8.033E-05   &  5.499   &    5.244E-05   &  5.452\\
  20   &    9.792E-07   & 11.723   &    4.143E-07   & 11.983 &    2.879E-06   & 11.570   &    1.928E-06   & 11.483\\
  25   &    1.045E-07   & 10.028   &    4.254E-08   & 10.201   &    3.149E-07   &  9.918   &    2.143E-07   &  9.844\\
  30   &    3.679E-09   & 18.355   &    1.378E-09   & 18.810  &    1.145E-08   & 18.179   &    7.944E-09   & 18.072\\
  35   &    3.971E-10   & 14.441   &    1.501E-10   & 14.384 &    1.259E-09   & 14.318   &    8.845E-10   & 14.240\\
  40   &    1.417E-11   & 24.964   &    5.056E-12   & 25.392 &    4.607E-11   & 24.774   &    3.281E-11   & 24.670\\
  45   &    1.566E-12   & 18.700   &    7.497E-13   & 16.205 &    5.106E-12   & 18.678   &    3.656E-12   & 18.632\\
  50   &    2.934E-13   & 15.891   &    2.746E-13   &  9.533 &    5.078E-13   & 21.906   &    4.288E-13   & 20.341\\
\hline
\hline
\end{tabular}
\end{table}

\begin{figure}[H]
\begin{minipage}{0.48\textwidth}
\centering
\includegraphics[width=3.2in]{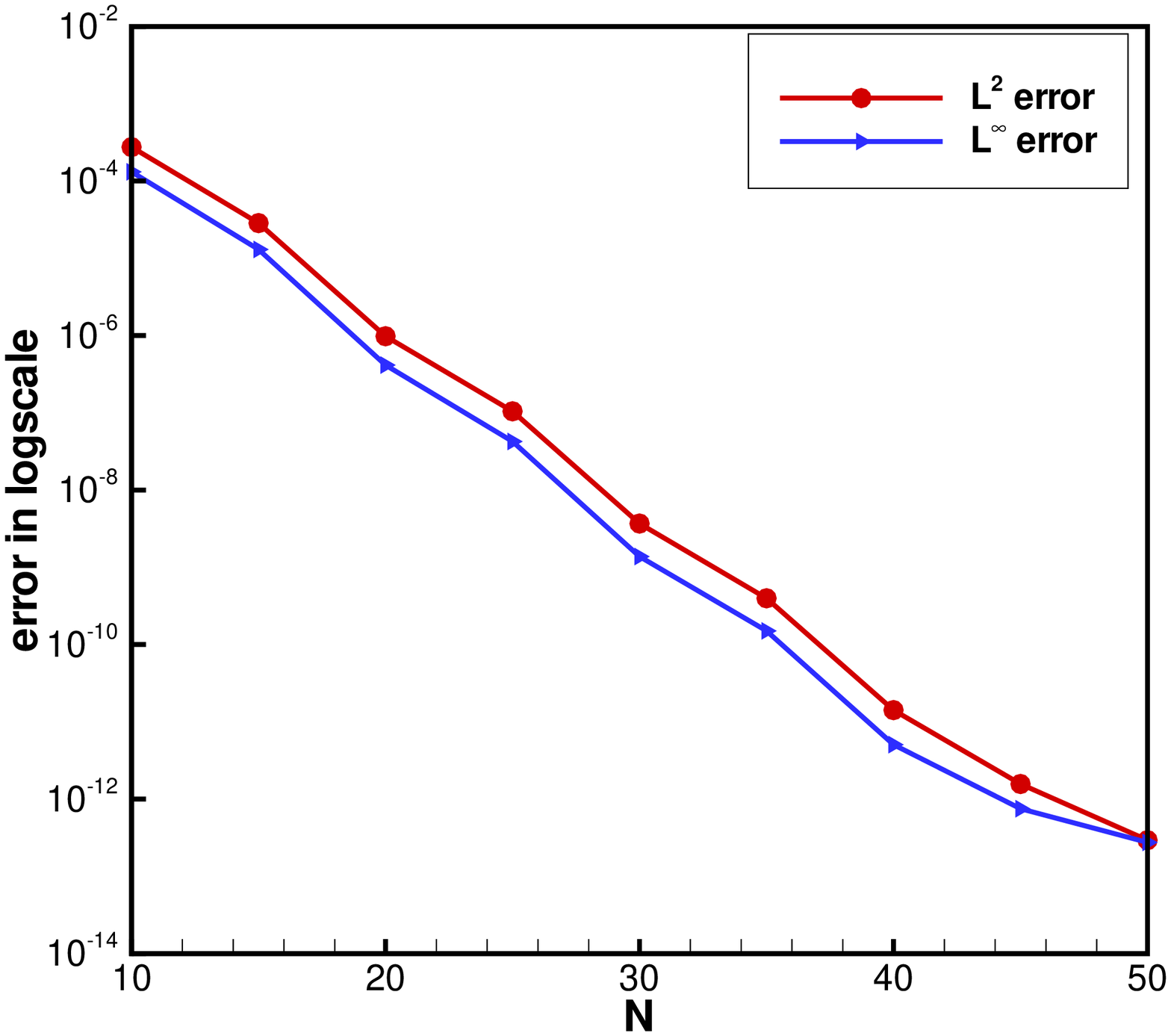}
\caption*{(a) case (\romannumeral1)}
\end{minipage}
\begin{minipage}{0.48\textwidth}
\centering
\includegraphics[width=3.2in]{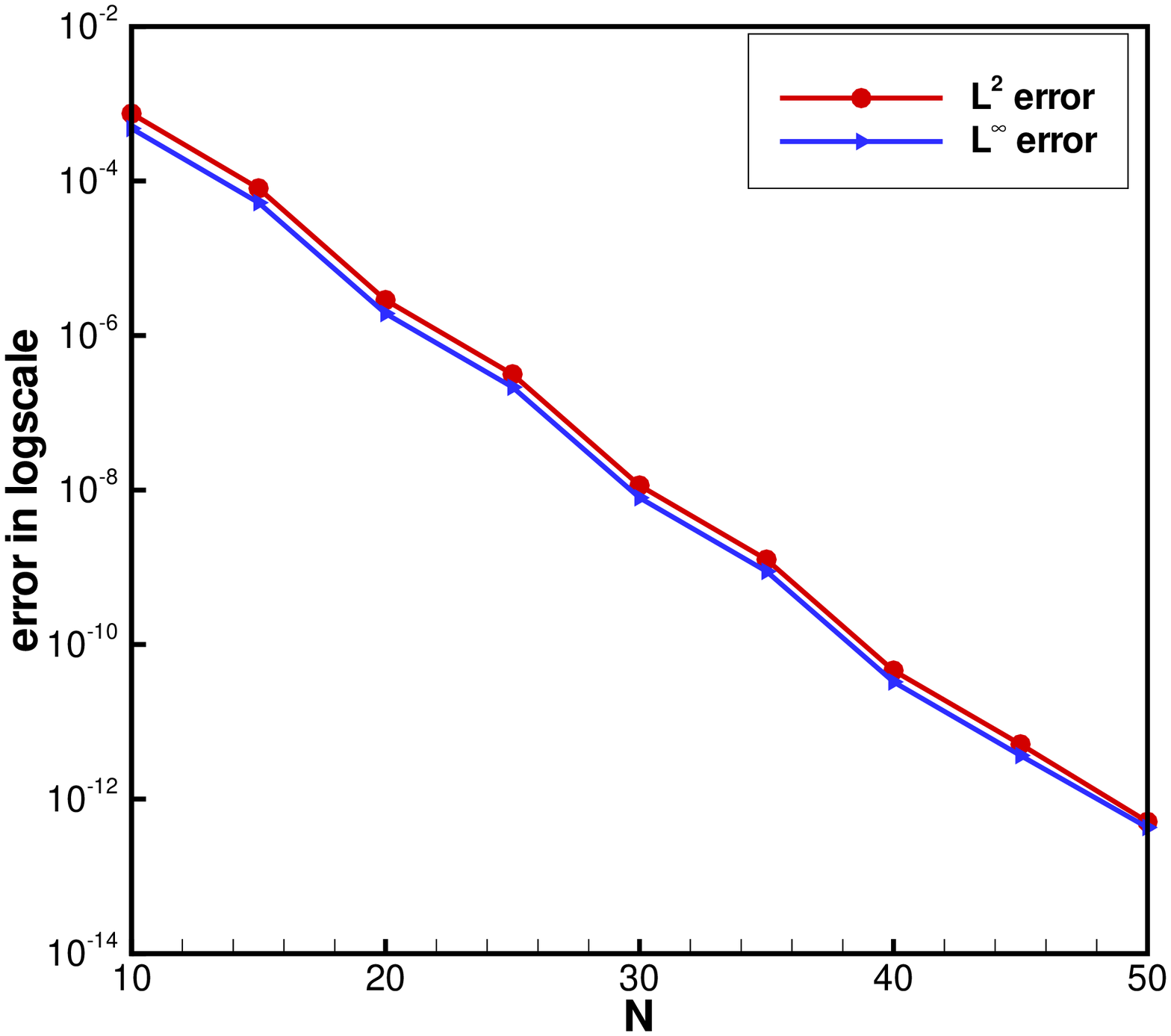}
\caption*{(b) case (\romannumeral2)}
\end{minipage}
\caption{Example 1: Convergence of the $L^2$ and $L^\infty$ errors for the solutions.}
\label{fig1}
\end{figure}

\textbf{Example 2}. Accuracy tests with \emph{smooth} initial conditions and source term in the two-dimensional case.
We compute the 2D diffusive-viscous wave equation
with $\alpha=\beta=\gamma=1$ and consider the following two cases:
\begin{enumerate}[(\romannumeral1)]
\item $f=0,~~~u(x,y,0)={\rm e}^{-(x^2+y^2)},~~~u_t(x,y,0)=-{\rm e}^{-(x^2+y^2)}.$
\item $f={\rm e}^{-x^2}[(3-4x^2-4y^2)\sin t+(5-4x^2-4y^2)\cos t],~~~u(x,y,0)=0,~~~u_t(x,y,0)={\rm e}^{-(x^2+y^2)}.$
\end{enumerate}

For these two problems, the exact solutions are given by
$u(x,y,t)={\rm e}^{-(x^2+y^2)-t}$ and
$u(x,y,t)={\rm e}^{-(x^2+y^2)}\sin t$, respectively. The degrees of the space approximation is $N\times N$.
Set $\Delta t_n=10^{-4}$, we compute the numerical approximations until
$T=0.5$. The errors are also measured by three different norms ($L^2$ and $L^\infty$). We present the errors as well as the convergence rates in Table \ref{tab3}
for the case (\romannumeral1) and the case (\romannumeral2), respectively. We also plot the convergence of the errors in semilog scale in Figure \ref{fig2} showing again that the spectral accuracy with respect to $N$ is obtained.
\begin{table}[H]
\centering
\caption{Example 2: Errors and convergence rates for the 2D cases.}
\label{tab3}
\begin{tabular}{c|c|c|c|c||c|c|c|c}
\hline
\hline
 & \multicolumn{4}{c||}{case (\romannumeral1)} & \multicolumn{4}{c}{case (\romannumeral2)}\\
 \hline
$N$ & $L^2$ error & $L^2$ order & $L^\infty$ error
& $L^\infty$ order & $L^2$ error & $L^2$ order & $L^\infty$ error & $L^\infty$ order\\
\hline
  10   &    7.181E-04   &  ---   &    4.293E-04   &  --- &   6.347E-04   &  ---   &    2.620E-04   &  ---\\
  15   &    7.452E-05   &  3.269   &    4.141E-05   &  3.374 &   6.781E-05   &  3.226   &    2.563E-05   &  3.354\\
  20   &    2.556E-06   &  8.318   &    1.302E-06   &  8.532 & 2.413E-06   &  8.227   &    9.349E-07   &  8.166\\
  25   &    2.728E-07   &  7.779   &    1.324E-07   &  7.947  &   2.630E-07   &  7.705   &    1.032E-07   &  7.660\\
  30   &    9.603E-09   & 14.997   &    4.373E-09   & 15.282 &   9.520E-09   & 14.872   &    3.808E-09   & 14.786\\
  35   &    1.037E-09   & 12.209   &    4.547E-10   & 12.415  &   1.045E-09   & 12.119   &    4.228E-10   & 12.056\\
  40   &    3.697E-11   & 21.626   &    1.562E-11   & 21.869  &   3.811E-11   & 21.480   &    1.575E-11   & 21.344\\
  45   &    4.028E-12   & 16.602   &    1.852E-12   & 15.969  &   4.202E-12   & 16.512   &    1.743E-12   & 16.483\\
  50   &    3.287E-13   & 21.275   &    2.891E-13   & 15.766  &   2.438E-13   & 24.173   &    1.822E-13   & 19.174\\
\hline
\hline
\end{tabular}
\end{table}

\begin{figure}[H]
\begin{minipage}{0.48\textwidth}
\centering
\includegraphics[width=3.2in]{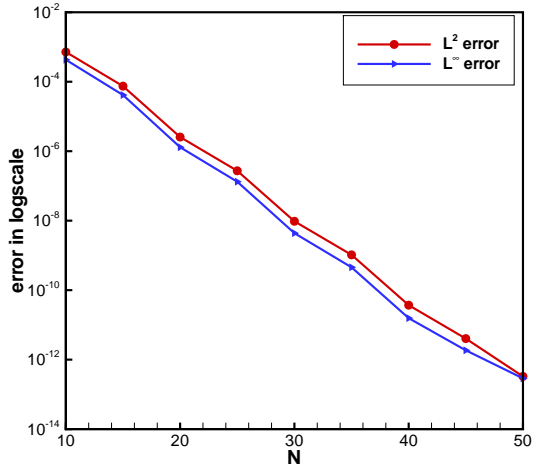}
\caption*{(a) case (\romannumeral1)}
\end{minipage}
\begin{minipage}{0.48\textwidth}
\centering
\includegraphics[width=3.2in]{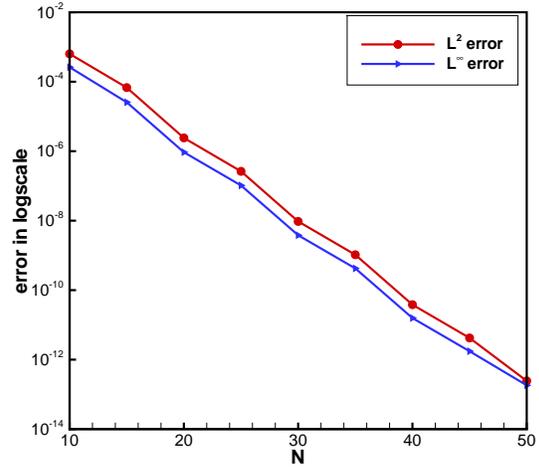}
\caption*{(b) case (\romannumeral2)}
\end{minipage}
\caption{Example 2: Convergence of the $L^2$ and $L^\infty$ errors for the solutions.}
\label{fig2}
\end{figure}

\textbf{Example 3}. Accuracy tests with \emph{non-smooth} source term in the one-dimensional case.
Here we consider the 1D problem with $\alpha=\beta=\gamma=1$
and the initial conditions and source function are taken as:
$$u(x,0)=u_t(x,0)=0,~~~f(x,t)=x^\mu{\rm e}^{-x^2}\cos(t).$$
For the above problem, we set $\mu=\frac{1}{3}$ and $\mu=\frac{4}{3}$, respectively. The expected convergence rates of $\|e\|_{H^1(\mathbb{R})}$ are almost $\frac{11}{12}$ and $\frac{17}{12}$, respectively according to the estimate \eqref{eq:errest2}. The numerical simulation is implemented with the present Hermite spectral method and the third-stage SSP Runge-Kutta method ($\Delta t_n=10^{-4}$) until the final time $T=0.5$. Since we don't have the exact solution, we take the numerical solutions obtained with $N=500$ as the ``reference solution'' to compute the numerical errors and consequent the convergence rates. The result concerning the convergence of the $H^2$-error is shown in Figure \ref{figadd}. We observe that the convergence rates are coincide with the theoretical results.
\begin{figure}[H]
\centering
\includegraphics[width=3.5in]{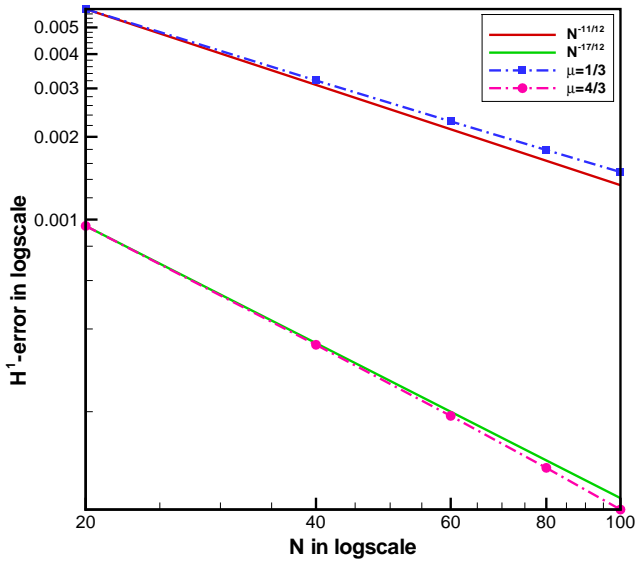}
\caption{Example 3: Convergence of the $H^1$ errors for the solutions.}
\label{figadd}
\end{figure}

\textbf{Example 4}. Wave propagation within homogeneous medium.
In this test, we present the problem describing the wave propagation with the homogeneous medium. For the diffusive-viscous wave equation \eqref{eq1}, we set the parameters as $\alpha=1,\beta=0.01$ and $\gamma=20$. A Ricker wavelet with dominant frequency of 15Hz located at $(x_0,y_0)=(10,10)$ is used to generate the vibration. The source function is taken as follows
\begin{equation}\label{source1}
f(x,y,t)=g(x,y)h(t),
\end{equation}
where
\begin{equation}\label{source2}
g(x,y)={\rm e}^{-[(x-x_0)^2+(y-y_0)^2]},~~~~h(t)=\big[1-2(\pi f_0(t-t_0))^2\big]{\rm e}^{-(\pi f_0(t-t_0))^2}
\end{equation}
with the dominant frequency $f_0=15$ and the time delay $t_0=0.05$.

We use the present algorithm to numerically solve this model. The time step size is taken as $\Delta t=10^{-4}$ and degrees of the space approximation is $N\times N$. We show the time evolution of the
diffusive-viscous wave in Figure \ref{fig3}. {Observe that the wave propagates outward
isotropically from the source center $(x_0,y_0)$.}
We further compare the cross sections of the numerical solutions at the line $y=x$ for $T=0.005,0.1,0.3,0.5$ with $N=100$ and $N=200$ in Figure \ref{fig4} to verify the convergence of our numerical method. It can be seen that the one-dimensional profiles match very well. Moreover, we notice that the wave front has propagated out of the fixed domain $[0,20]^2$ (here only for showing the solution) at the time $T=0.5$. This means that we cannot obtain accurate solutions for a long time within a fixed bounded domain, e.g. $[0,20]^2$, if homogeneous boundary conditions are used.
\begin{figure}[H]
\begin{minipage}{0.5\textwidth}
\centering
\includegraphics[width=3.2in]{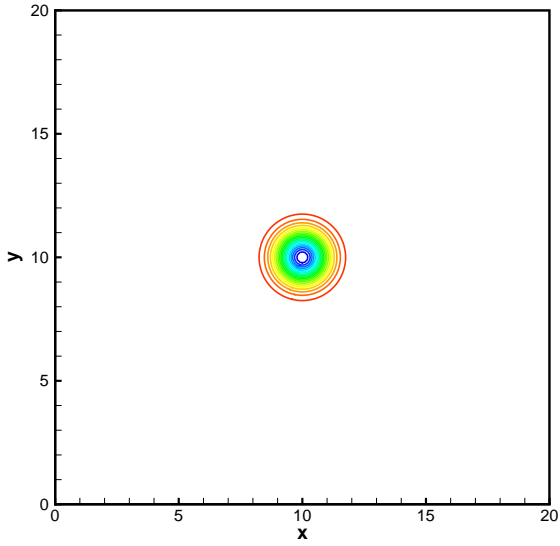}
\caption*{(a) $T=0.005$}
\end{minipage}
\begin{minipage}{0.5\textwidth}
\centering
\includegraphics[width=3.2in]{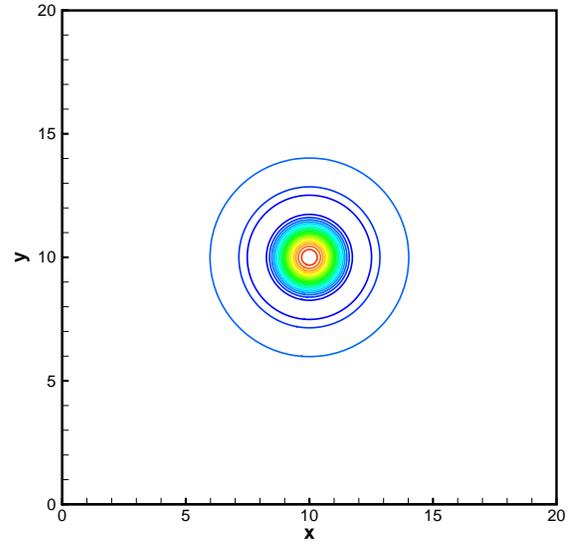}
\caption*{(b) $T=0.1$}
\end{minipage}
\begin{minipage}{0.5\textwidth}
\centering
\includegraphics[width=3.2in]{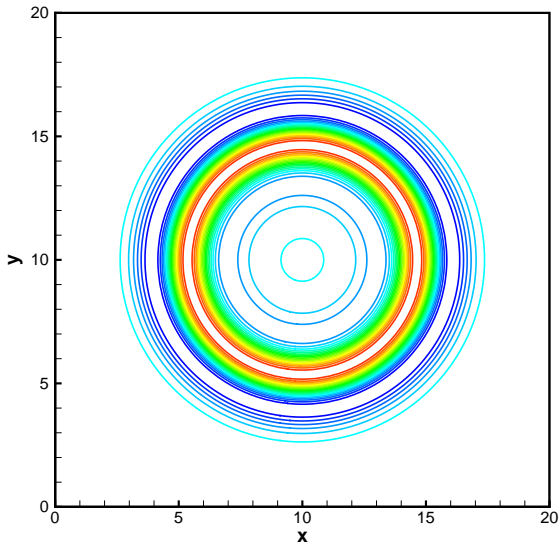}
\caption*{(c) $T=0.3$}
\end{minipage}
\begin{minipage}{0.5\textwidth}
\centering
\includegraphics[width=3.2in]{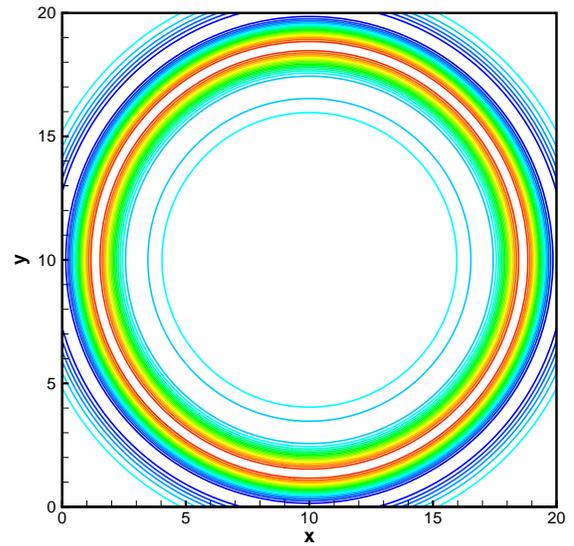}
\caption*{(d) $T=0.5$}
\end{minipage}
\caption{Example 4: Contours of the numerical approximations at time $T=0.005,0.1,0.3,0.5$ with $N=200$.}
\label{fig3}
\end{figure}
\begin{figure}[H]
\begin{minipage}{0.5\textwidth}
\centering
\includegraphics[width=3.0in]{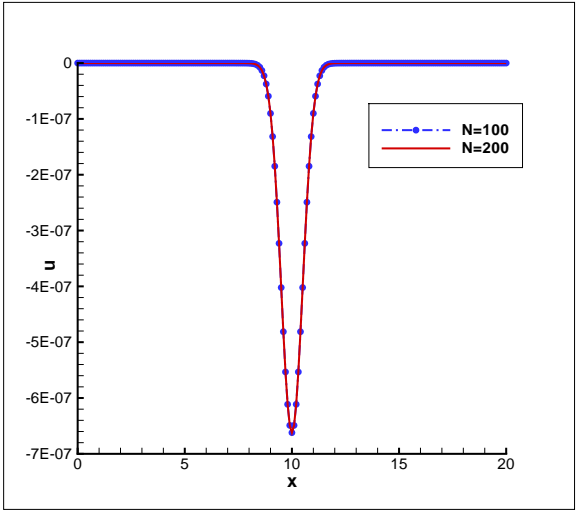}
\caption*{(a) $T=0.005$}
\end{minipage}
\begin{minipage}{0.5\textwidth}
\centering
\includegraphics[width=3.0in]{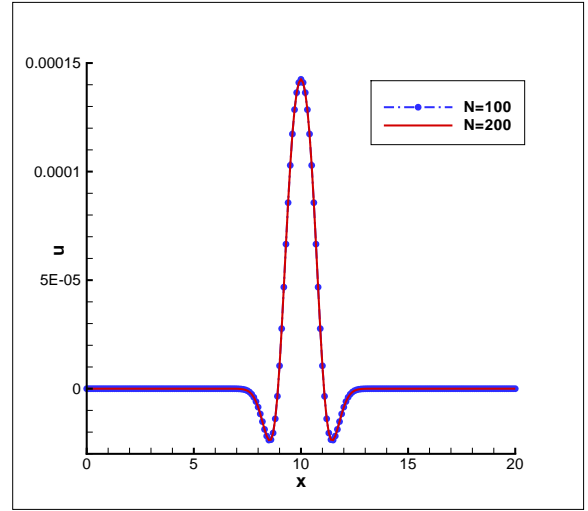}
\caption*{(b) $T=0.1$}
\end{minipage}
\begin{minipage}{0.5\textwidth}
\centering
\includegraphics[width=3.0in]{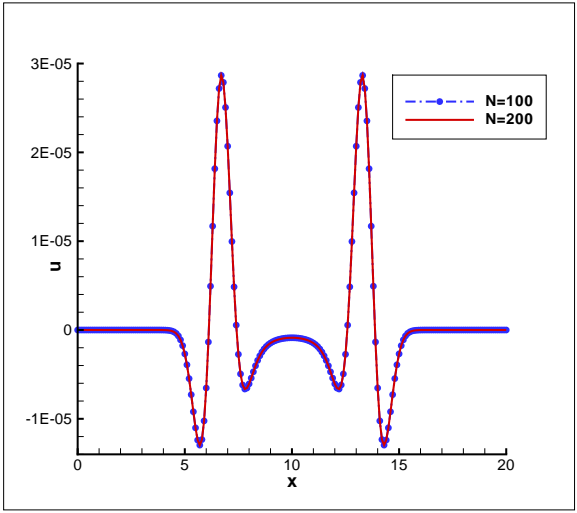}
\caption*{(c) $T=0.3$}
\end{minipage}
\begin{minipage}{0.5\textwidth}
\centering
\includegraphics[width=3.0in]{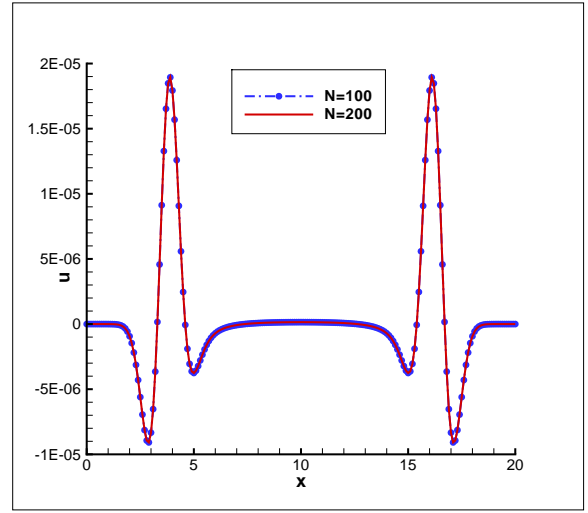}
\caption*{(d) $T=0.5$}
\end{minipage}
\caption{Example 4: Comparison of the cross sections of the numerical solutions at $y=x$,
at $T=0.005,0.1,0.3,0.5$. The results corresponding to red solid lines are obtained by using $N=200$ while the results corresponding to blue
dash-dotted lines with the circle symbol are obtained by using $N=100$.}
\label{fig4}
\end{figure}

\textbf{Example 5}. Wave propagation within heterogeneous media.
Now we consider the wave propagation within two different  media.  The parameters are set as
$$(\alpha,\beta,\gamma)=\left\{\begin{array}{ll}
(1.0,0.02,15.6), &~\text{if}~y\le 16.5,\\
(2.5,0.05,20.4), &~\text{if}~y>16.5.
\end{array}\right.$$
The source function is defined as that in \eqref{source1} and \eqref{source2} with
$(x_0,y_0)=(15,15),f_0=20$ and $t_0=0.05$.

The time step size is again taken as $\Delta t=10^{-4}$ and degrees of the space approximation is $N\times N$. We present the time
evolution of the diffusive-viscous wave in Figure \ref{fig5}. We observe that initially the wave propagates
isotropically from the source center $(x_0,y_0)$ until it reaches the interface of these different two media (around $T=0.05$ to 0.15).
Then at a later time the wave fronts propagate at different speeds within these two media. We also present the cross sections of the numerical solutions at the line $x=17$ for $T=0.05,0.15,0.25,0.4$ with $N=150$ and $N=300$ in Figure \ref{fig6} to verify the convergence of our numerical method. It can be seen that the one-dimensional profiles match very well. Again, it can be observed that the wave fronts have propagated out of the given domain $[0,30]^2$ at the time $T=0.8$. For this case, suitable boundary conditions must be applied when the problem is simulated in a bounded domain, otherwise truncation errors  or boundary reflections may destroy the numerical solutions. However, this issue is resolved by using the proposed method since we directly simulate the model in natural unbounded domains.
\begin{figure}[H]
\begin{minipage}{0.5\textwidth}
\centering
\includegraphics[width=3.6in]{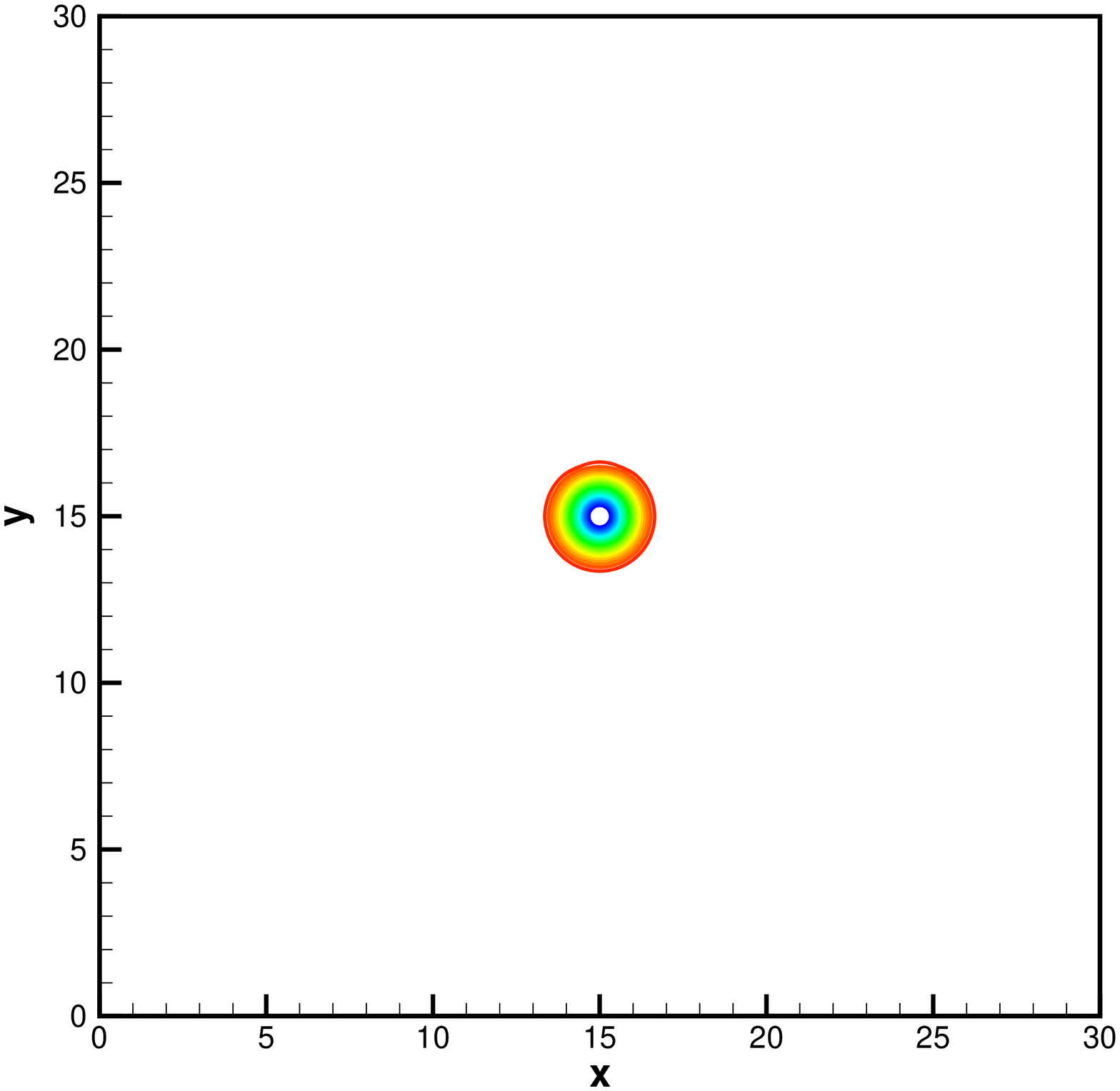}
\caption*{(a) $T=0.05$}
\end{minipage}
\begin{minipage}{0.5\textwidth}
\centering
\includegraphics[width=3.6in]{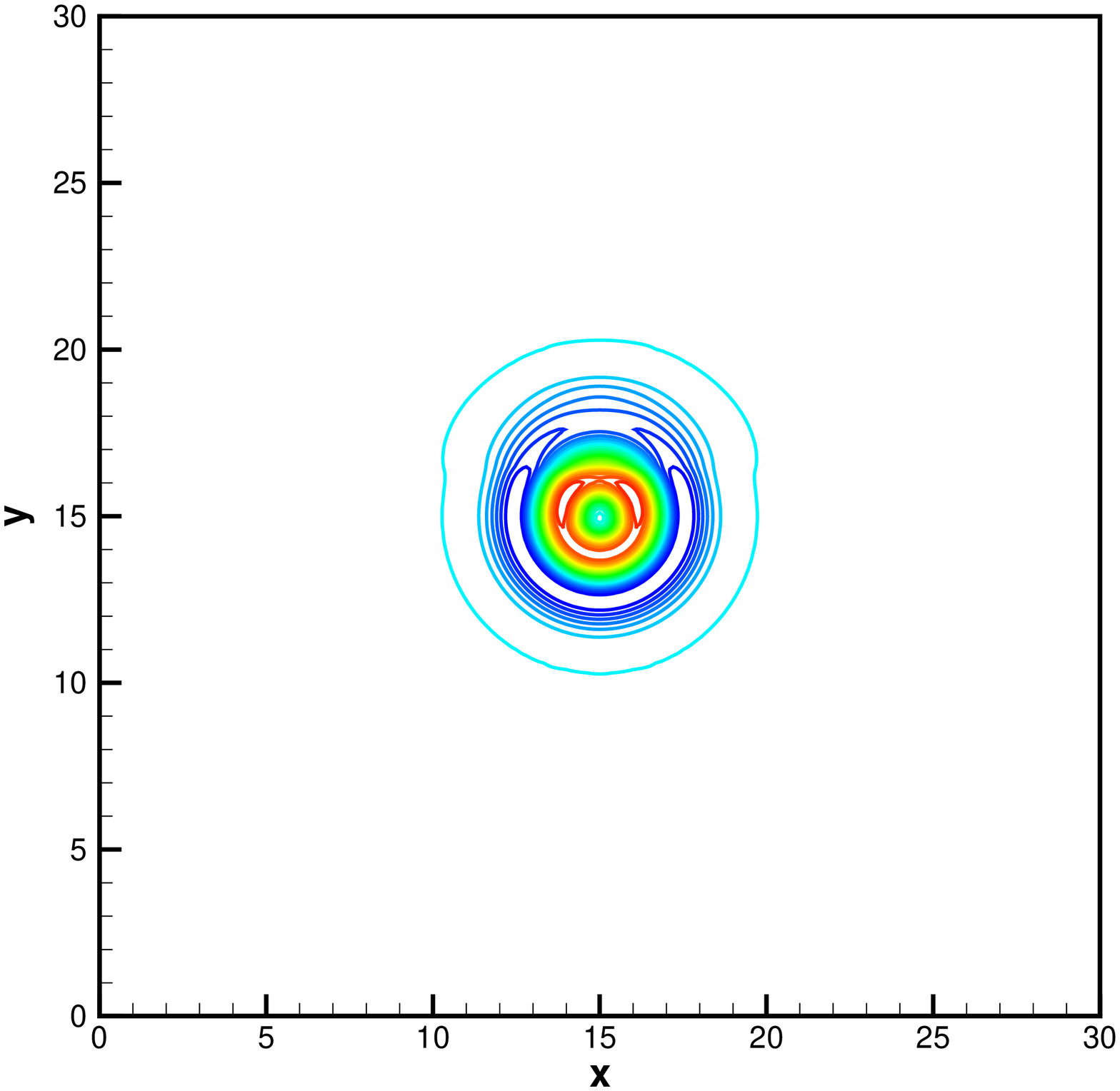}
\caption*{(b) $T=0.15$}
\end{minipage}
\end{figure}
\begin{figure}[H]
\begin{minipage}{0.5\textwidth}
\centering
\includegraphics[width=3.6in]{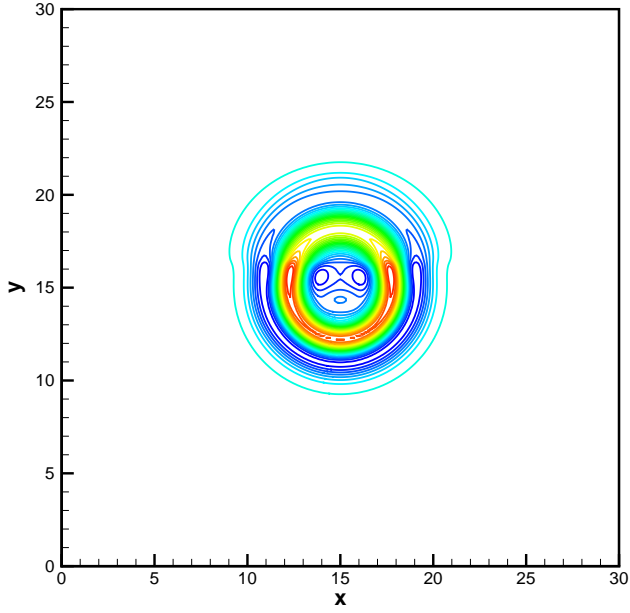}
\caption*{(c) $T=0.25$}
\end{minipage}
\begin{minipage}{0.5\textwidth}
\centering
\includegraphics[width=3.6in]{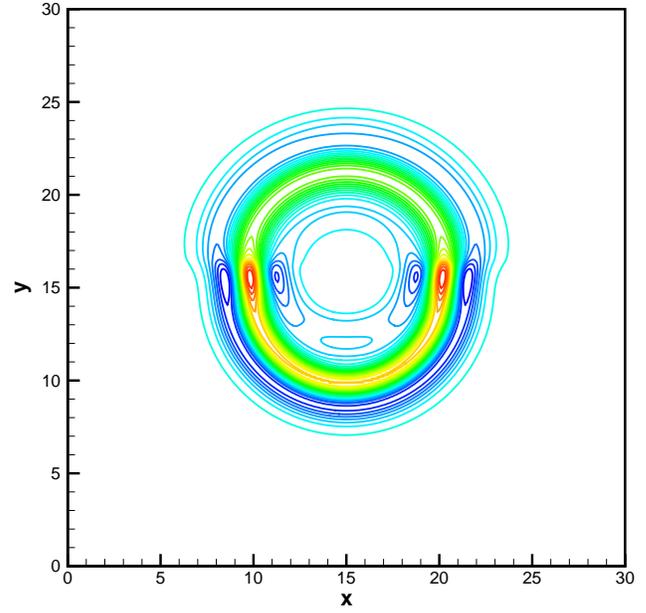}
\caption*{(d) $T=0.4$}
\end{minipage}
\begin{minipage}{0.5\textwidth}
\centering
\includegraphics[width=3.6in]{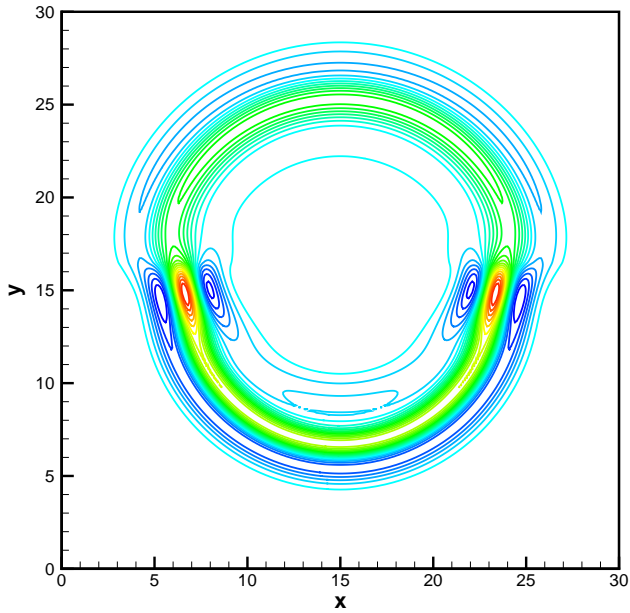}
\caption*{(e) $T=0.6$}
\end{minipage}
\begin{minipage}{0.5\textwidth}
\centering
\includegraphics[width=3.6in]{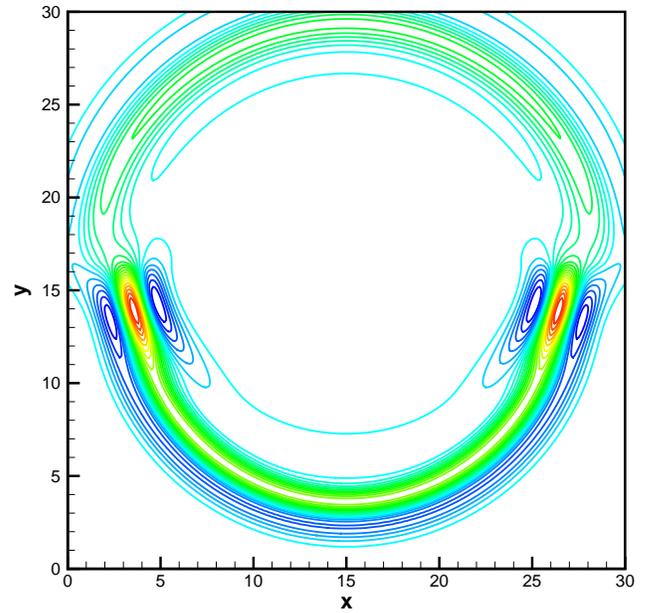}
\caption*{(f) $T=0.8$}
\end{minipage}
\caption{Example 5: Contours of the numerical approximations at time $T=0.05,0.15,0.25,0.4,0.6,0.8$ with $N=300$.}
\label{fig5}
\end{figure}
\begin{figure}[H]
\begin{minipage}{0.5\textwidth}
\centering
\includegraphics[width=2.6in]{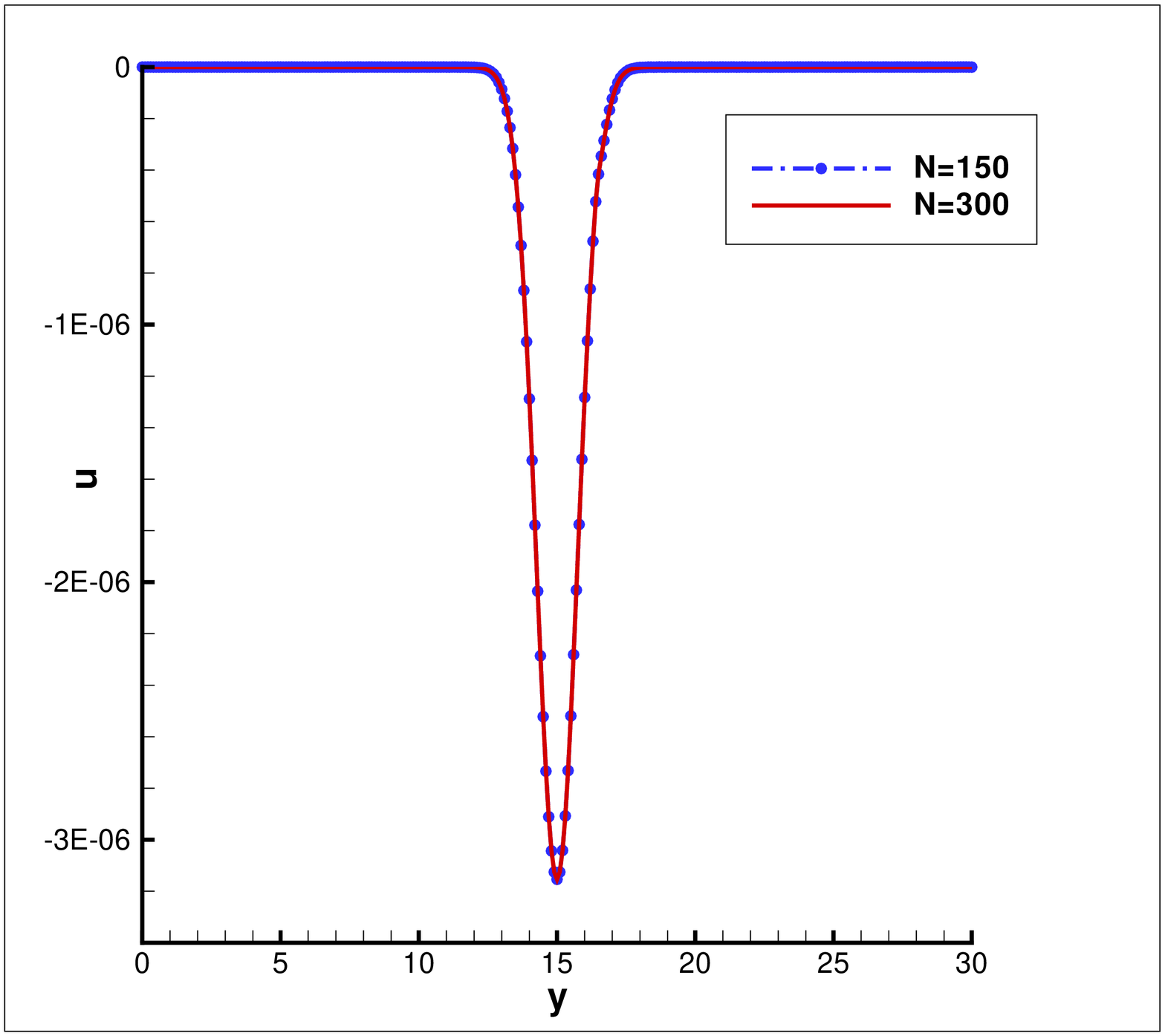}
\caption*{(a) $T=0.05$}
\end{minipage}
\begin{minipage}{0.5\textwidth}
\centering
\includegraphics[width=2.6in]{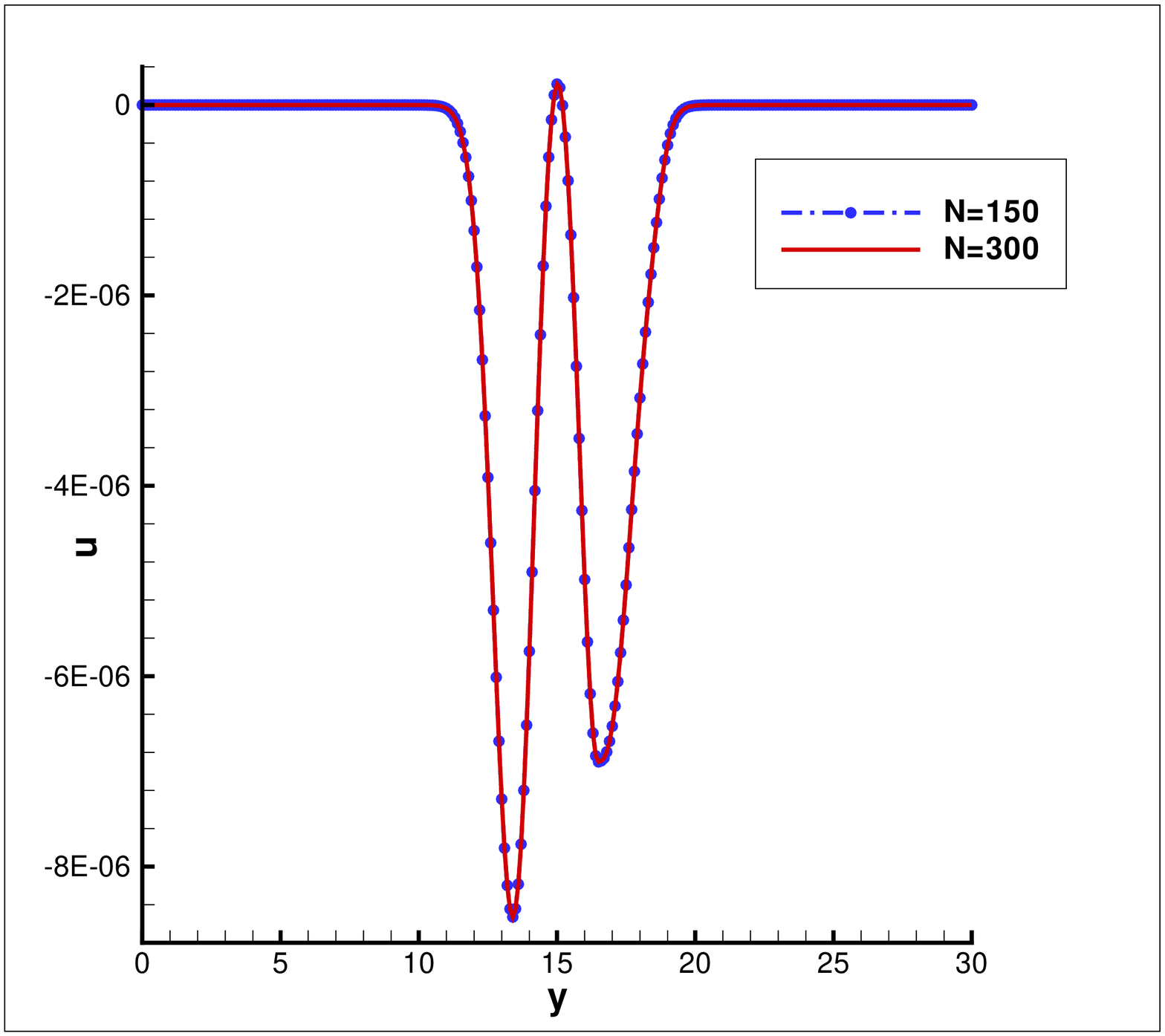}
\caption*{(b) $T=0.15$}
\end{minipage}
\begin{minipage}{0.5\textwidth}
\centering
\includegraphics[width=2.6in]{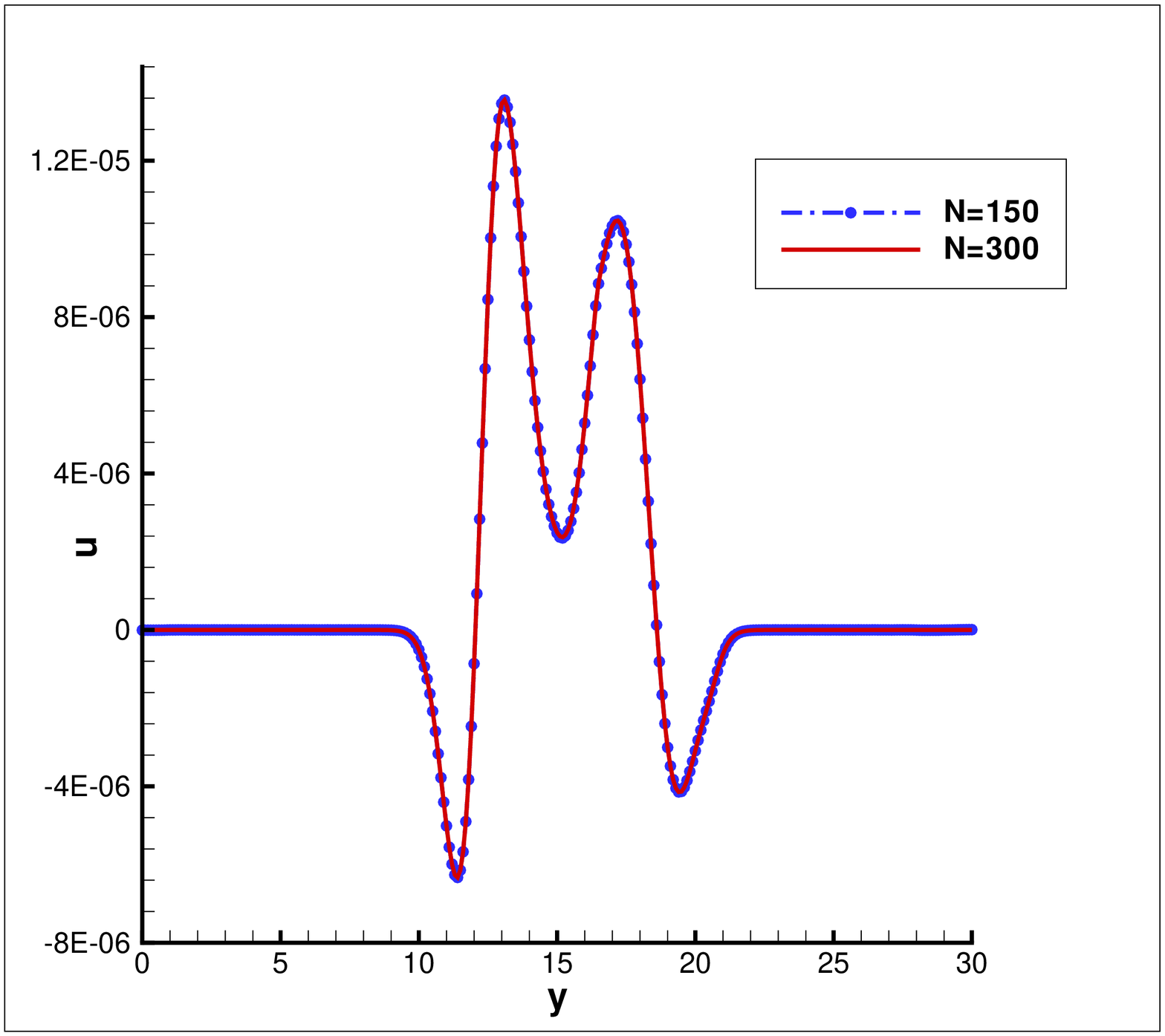}
\caption*{(c) $T=0.25$}
\end{minipage}
\begin{minipage}{0.5\textwidth}
\centering
\includegraphics[width=2.6in]{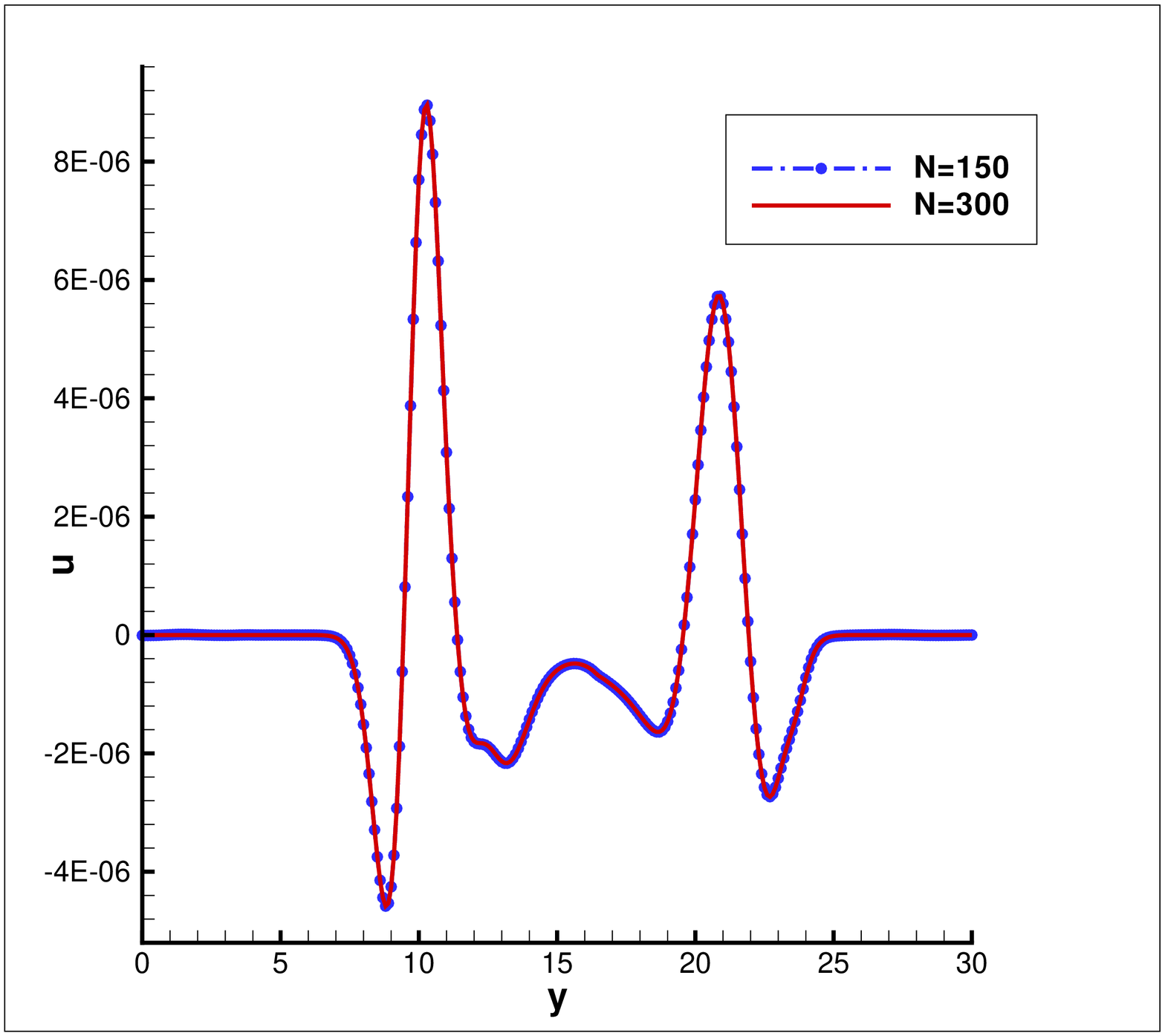}
\caption*{(d) $T=0.4$}
\end{minipage}
\begin{minipage}{0.5\textwidth}
\centering
\includegraphics[width=2.6in]{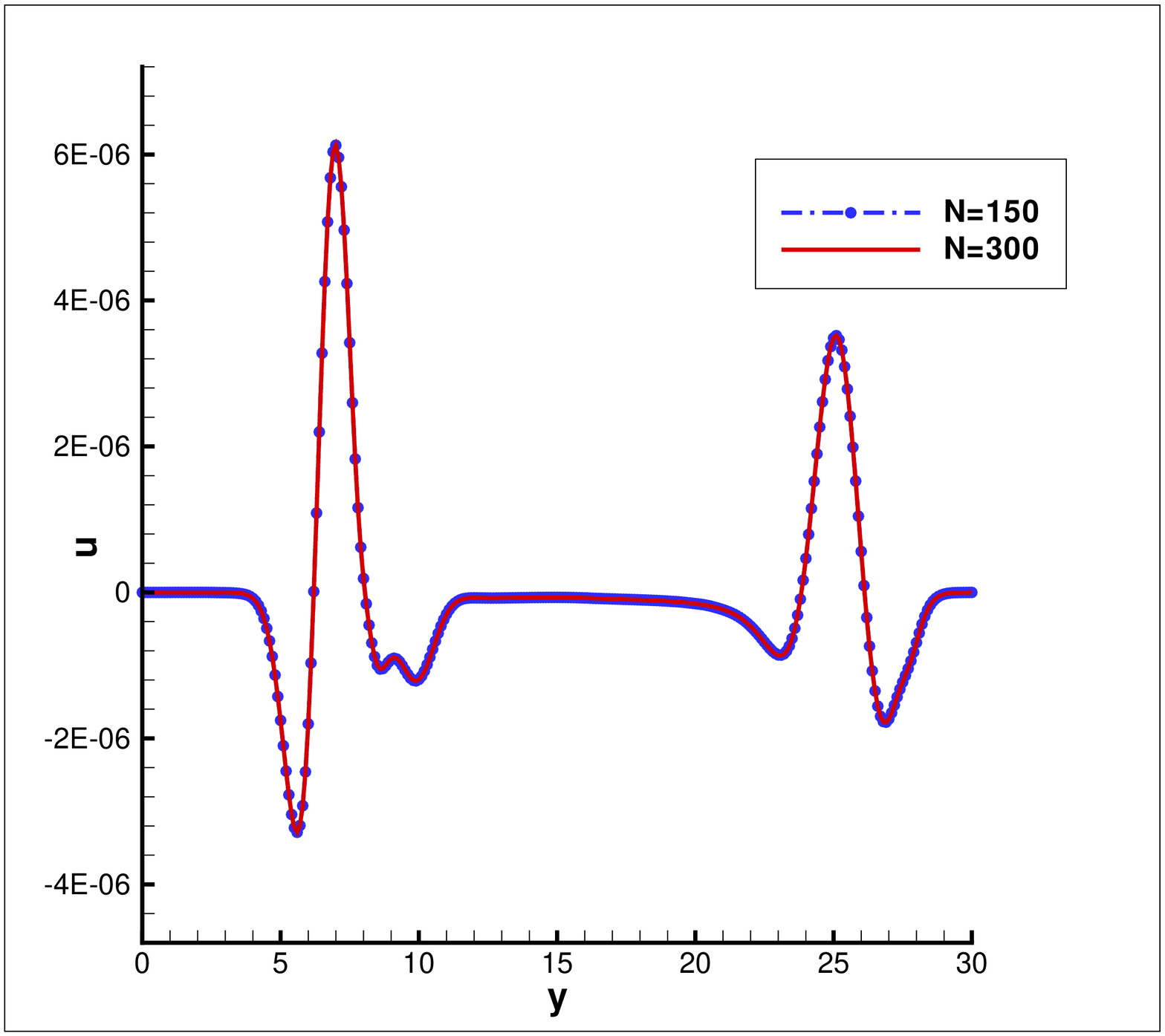}
\caption*{(e) $T=0.6$}
\end{minipage}
\begin{minipage}{0.5\textwidth}
\centering
\includegraphics[width=2.6in]{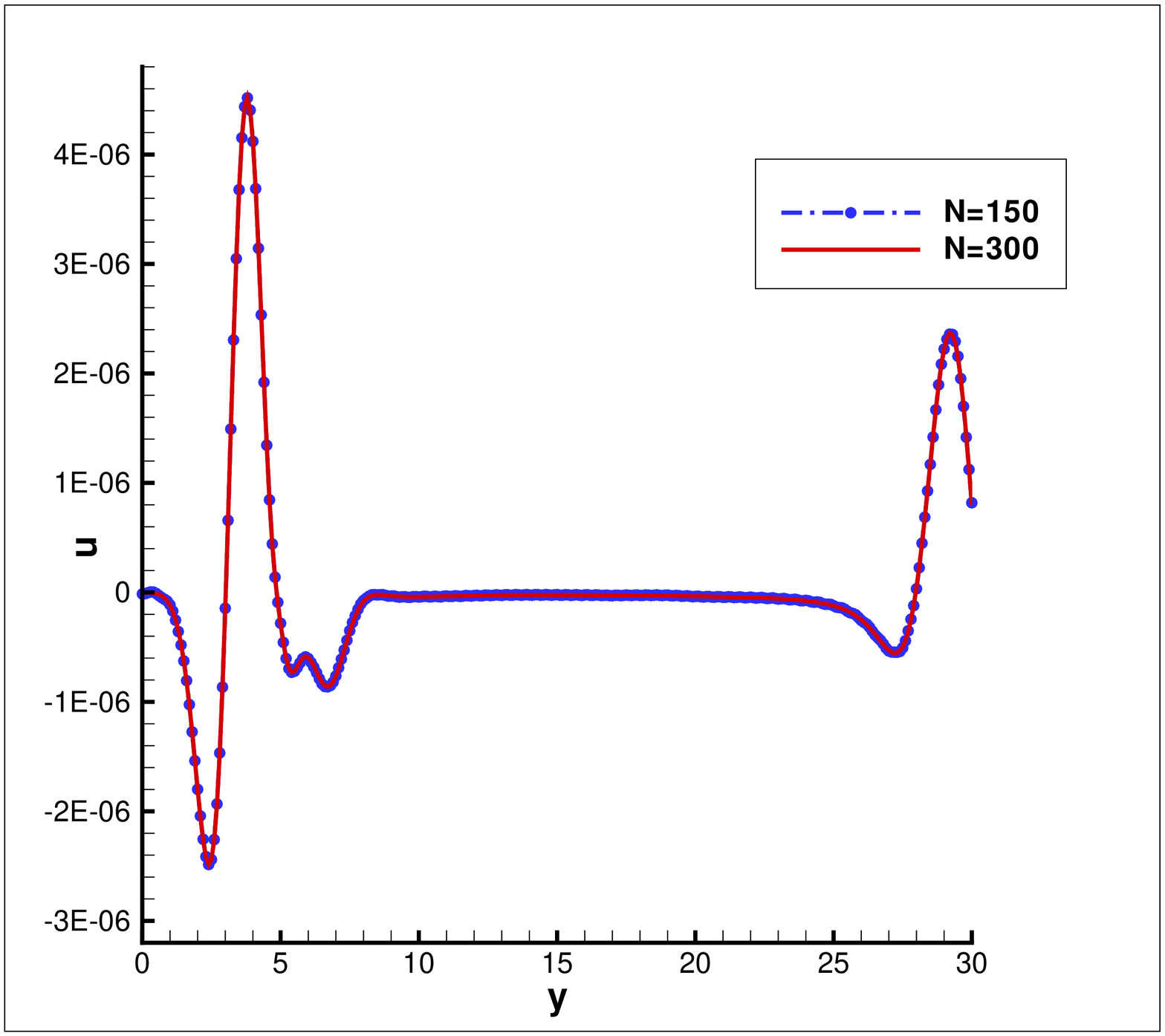}
\caption*{(f) $T=0.8$}
\end{minipage}
\caption{Example 5: Comparison of the cross sections of the numerical solutions with $x=17$,
at $T=0.05,0.15,0.25,0.4,0.6,0.8$. The results corresponding to red solid lines are obtained by using $N=300$ while the results corresponding to blue
dash-dotted lines with the circle symbol are obtained by using $N=150$.}
\label{fig6}
\end{figure}

\section{Concluding remarks}
Diffusive-viscous wave equations arising in geophysics are naturally developed in unbounded domains. A truncated domain is usually needed to numerically solve the diffusive-viscous wave equations. However, this introduces nonphysical reflections or truncation errors. To resolve this issue, we directly consider diffusive-viscous wave equations in unbounded domains in this paper. In particular, we analyzed the existence and uniqueness of the weak solution and show the regularity in terms of the initial conditions and the source term. We further developed a high accuracy Hermite spectral Galerkin scheme for diffusive-viscous wave equations, and then derived the error estimate for the Hermite spectral Galerkin method. We demonstrated the theoretical result and verified the sharpness of the error estimate using both smooth and non-smooth functions $f$. We further provided several numerical examples with constant as well as discontinuous coefficients to demonstrate the present algorithm showing that the present method can resolve the boundary truncation and artificial reflection issues.

\end{document}